\documentclass[a4paper,12pt]{article}
\let\Horig\H
\usepackage[english]{babel}
\usepackage{tikz}
\usetikzlibrary{matrix,arrows,decorations.markings}
\usepackage{amsmath,amsfonts,amssymb,amsthm,url,textcomp}
\usepackage{csquotes}
\usepackage{a4wide}
\usepackage[numbers]{natbib}
\usepackage{fancyhdr}
\usepackage{anyfontsize}
\usepackage{hyperref}
\usepackage{hhline}
\usepackage{leftidx}

\allowdisplaybreaks

\newtheorem{theorem}{Theorem}\numberwithin{theorem}{section}
\newtheorem{definition}[theorem]{Definition}
\newtheorem{lemma}[theorem]{Lemma}
\newtheorem{corollary}[theorem]{Corollary}
\newtheorem{proposition}[theorem]{Proposition}

\newtheorem{question}[theorem]{Question}

\newtheorem{theoremm}{Theorem}\numberwithin{theoremm}{subsection}
\newtheorem{deffinition}[theoremm]{Definition}

\numberwithin{theoremmm}{subsubsection}

\theoremstyle{remark}
\newtheorem{remark}[theorem]{Remark}

\newtheorem{remmark}[theoremm]{Remark}

\newcommand{\Rad}{\operatorname{Rad}}
\newcommand{\Aut}{\operatorname{Aut}}
\newcommand{\Alt}{\mathcal{A}}
\newcommand{\PSL}{\operatorname{PSL}}

\newcommand{\lcm}{\operatorname{lcm}}

\newcommand{\ord}{\operatorname{ord}}
\newcommand{\Sym}{\mathcal{S}}
\newcommand{\Hol}{\operatorname{Hol}}

\newcommand{\C}{\mathcal{C}}

\newcommand{\Soc}{\operatorname{Soc}}

\newcommand{\id}{\operatorname{id}}

\newcommand{\fix}{\operatorname{fix}}
\newcommand{\Out}{\operatorname{Out}}

\newcommand{\PGL}{\operatorname{PGL}}
\newcommand{\Gal}{\operatorname{Gal}}
\newcommand{\GL}{\operatorname{GL}}

\newcommand{\cha}{\operatorname{char}}

\renewcommand{\P}{\operatorname{P}}

\newcommand{\f}{\operatorname{f}}

\newcommand{\p}{\operatorname{p}}

\newcommand{\IN}{\mathbb{N}}

\newcommand{\IF}{\mathbb{F}}

\newcommand{\IZ}{\mathbb{Z}}

\newcommand{\R}{\mathcal{R}}

\renewcommand{\S}{\mathfrak{S}}
\newcommand{\h}{\operatorname{h}}
\newcommand{\Part}{\operatorname{Part}}
\newcommand{\Suz}{\operatorname{Suz}}

\begin{document}

\title{Fibers of word maps and the multiplicities of nonabelian composition factors}

\author{Alexander Bors\thanks{University of Salzburg, Mathematics Department, Hellbrunner Stra{\ss}e 34, 5020 Salzburg, Austria. \newline E-mail: \href{mailto:alexander.bors@sbg.ac.at}{alexander.bors@sbg.ac.at} \newline The author is supported by the Austrian Science Fund (FWF):
Project F5504-N26, which is a part of the Special Research Program \enquote{Quasi-Monte Carlo Methods: Theory and Applications}. \newline 2010 \emph{Mathematics Subject Classification}: Primary: 20D05, 20D60, 20F70. Secondary: 20D30, 20G40. \newline \emph{Key words and phrases:} Finite groups, Word maps, Composition factors, Coset identities.}}

\date{\today}

\maketitle

\abstract{Call a reduced word $w$ \emph{multiplicity-bounding} if and only if a finite group on which the word map of $w$ has a fiber of positive proportion $\rho$ can only contain each nonabelian finite simple group $S$ as a composition factor with multiplicity bounded in terms of $\rho$ and $S$. In this paper, based on recent work of Nikolov, we present methods to show that a given reduced word is multiplicity-bounding and apply them to give some nontrivial examples of multiplicity-bounding words, such as words of the form $x^e$, where $x$ is a single variable and $e$ an odd integer.}

\section{Introduction}\label{sec1}

\subsection{Motivation and main results}\label{subsec1P1}

In recent years, there has been an intense interest in word maps on groups, particularly in their fibers and images; interested readers will find a good overview of the developments and results of this area in the survey article \cite{Sha13a}.

Recall that a (reduced) word in $d$ variables $X_1,\ldots,X_d$ is an element of the free group $F(X_1,\ldots,X_d)$. To each such word $w$ and each group $G$, one can associate a word map $w_G:G^d\rightarrow G$, induced by substitution. Studying the fibers of $w_G$ means studying the solution sets in $G^d$ to equations of the form $w=w(X_1,\ldots,X_d)=g$ for $g\in G$ fixed.

Viewing $w$ as a fixed, an interesting question is what one can say about finite groups $G$ where $w_G$ has a \enquote{large} fiber, say of positive proportion within $G^d$ (i.e., of size at least $\rho|G|^d$ for some $\rho\in\left(0,1\right]$ also fixed beforehand). Recently, building on earlier work of Larsen and Shalev from \cite{LS12a}, the author showed that such an assumption rules out large alternating groups and simple Lie type groups of large rank as composition factors of $G$, see \cite[Theorem 1.1.2]{Bor17b}.

One may view this result as providing an upper bound (namely $0$) on the multiplicity with which a nonabelian finite simple group of one of those two types can occur as a composition factor of $G$ under the assumption on the maximum fiber size of $w_G$. The aim of this paper is to provide some methods to show that a given word $w$ is \enquote{multiplicity-bounding} in general, in the precise sense of Definition \ref{multBoundDef} below, and to apply them to give some interesting examples of multiplicity-bounding words.

\begin{deffinition}\label{multBoundDef}
Let $w$ be a reduced word in $d$ distinct variables, and denote by $\S$ the collection of nonabelian finite simple groups. We say that $w$ is \emph{multiplicity-bounding} if and only if there exists a function $\beta_w:\left(0,1\right]\times\S\rightarrow\IN$ such that for all $(\rho,S)\in\left(0,1\right]\times\S$, the following holds: If $G$ is a finite group such that $w_G$ has a fiber of size at least $\rho|G|^d$, then the multiplicity of $S$ as a composition factor of $G$ is at most $\beta_w(\rho,S)$.
\end{deffinition}

The main results of this paper are summarized in the following theorem:

\begin{theoremm}\label{mainTheo}
The following reduced words are multiplicity-bounding:

\begin{enumerate}
\item Words of the form $x^e$ where $x$ is a single variable and either $e$ is odd or $e$ is even and $22\geq|e|\notin\{8,12,16,18\}$; moreover, the words $x^{\pm e}$ for $e\in\{8,12,16,18,24,30\}$ are \emph{not} multiplicity-bounding.
\item The words $\gamma_n(x_1,\ldots,x_n)$, $n$ a positive integer, defined recursively via $\gamma_1:=x_1$ and $\gamma_{n+1}:=[x_{n+1},\gamma_n]$ (commutator brackets).
\item All nonempty reduced words of length at most $8$ except for the words $x^{\pm 8}$, $x$ a single variable, which are not multiplicity-bounding by statement (1).
\end{enumerate}
\end{theoremm}

Henceforth, let us call a reduced word of the form $x^e$, $x$ a single variable and $e\in\IZ$, a \emph{power word}.

\begin{remmark}\label{mainRem}
Some comments on Theorem \ref{mainTheo}:

\begin{enumerate}
\item By Theorem \ref{mainTheo}(1), we know in particular for all integers $e$ with $|e|\leq 24$ whether the power word $x^e$ is multiplicity-bounding.
\item The finite list of even exponents with multiplicity-bounding power word in Theorem \ref{mainTheo}(1) is most likely not exhaustive; in Section \ref{sec5}, we will discuss an algorithm deciding for a given integer $e$ whether the power word $x^e$ is multiplicity-bounding, and the list in Theorem \ref{mainTheo}(1) was obtained by the author by applying an implementation of that algorithm in GAP \cite{GAP4} to those $e$. It is open whether there are infinitely many even integers $e$ such that $x^e$ is multiplicity-bounding, see also Question \ref{evenQues}.
\item Theorem \ref{mainTheo}(3) is proved using reduction arguments, to break down the number of words to check, combined with extensive computer calculations with GAP. We note that all GAP source code used in the production of this paper is available on the author's personal homepage, see \url{https://alexanderbors.wordpress.com/sourcecode/FibMult}; alternatively, a text file with the source code is also available from the author upon request.
\item By Theorem \ref{mainTheo}(3), we know in particular that all nonempty reduced words of length at most $7$ are multiplicity-bounding, and that $7$ is the optimal integer constant in that statement.
\item The proofs of the parts of Theorem \ref{mainTheo} are scattered across the paper. For statement (1), see Corollaries \ref{oddCorollary}, \ref{evenCorollary1} and \ref{evenCorollary2}. For statement (2), see Corollary \ref{examplesCor} and Proposition \ref{stronglyProp}(1,2). For statement (3), see Section \ref{sec6}.
\end{enumerate}
\end{remmark}

\subsection{Overview of the paper}\label{subsec1P2}

We now give an overview of Sections \ref{sec2} to \ref{sec7} of this paper.

In Section \ref{sec2}, we discuss a sufficient property, called \enquote{very strongly multiplicity-bounding}, for a reduced word to be multiplicity-bounding based on so-called \enquote{coset word maps}. This is elementary, but a bit technical; the key ideas are that the assumption that $w_G$ has a fiber of positive proportion transfers to $w_{G/\Rad(G)}$, and since $G/\Rad(G)$ acts faithfully by conjugation on its socle, which is a direct product of nonabelian finite simple groups, one is naturally led to studying certain equations over powers $S^n$, $S$ a nonabelian finite simple group, see Lemma \ref{equationSystemLem}.

Some simple methods to prove that a given word is very strongly multiplicity-bounding are discussed in Section \ref{sec3}, and this will already be enough to give a proof of Theorem \ref{mainTheo}(2).

In Section \ref{sec4}, we build on deeper results, most notably Nikolov's recent result \cite[Proposition 7]{Nik16a} (concerning coset identities over nonabelian finite simple groups) and information on conjugacy classes of maximal subgroups of finite simple groups to heavily reduce the collection of simple groups which one must check in order to prove that a given word is very strongly multiplicity-bounding, making it possible to perform these checks for short words in a reasonable amount of time with a modern computer, which will be used later in Section \ref{sec6}.

Section \ref{sec5} is devoted to power words. We will see that the general theory developed so far simplifies a lot when specialized to them, allowing for a direct decision algorithm whether a given power word is multiplicity-bounding and further reductions, which we use to check power words up to length $24$ with a computer.

In Section \ref{sec6}, we systematically study short words (i.e., words of length up to $8$). We first prove some further reduction results and apply them directly to cut down the lists of short words to check as far as possible. After this, we apply a GAP implementation of an algorithm written by the author to actually check that all words of length up to $8$ except for the power words of length $8$ are very strongly multiplicity-bounding, in particular multiplicity-bounding.

Finally, Section \ref{sec7} presents some open problems for further research.

\subsection{Some notation used throughout the paper}\label{subsec1P3}

We denote by $\IN$ the set of natural numbers (including $0$) and by $\IN^+$ the set of positive integers. The image and preimage of a set $M$ under a function $f$ are denoted by $f[M]$ and $f^{-1}[M]$ respectively.

For groups $G_1$ and $G_2$, we denote their free product by $G_1 \ast G_2$, and we write $G_1\leq G_2$ for \enquote{$G_1$ is a subgroup of $G_2$}, $G_1\unlhd G_2$ for \enquote{$G_1$ is a normal subgroup of $G_2$}, and $G_1\cha G_2$ for \enquote{$G_1$ is a characteristic subgroup of $G_2$}. The symmetric and alternating group on the finite set $\{1,\ldots,n\}$ are denoted by $\Sym_n$ and $\Alt_n$ respectively. For an odd positive integer $O$, the Suzuki group over $\IF_{2^O}$ (not the sporadic Suzuki group) is denoted by $\Suz(2^O)$. For an element $g$ of a group $G$, we denote by $\tau_g:G\rightarrow G$, $x\mapsto gxg^{-1}$, the conjugation by $g$ on $G$. The holomorph of a group $G$, i.e., the semidirect product $G\rtimes\Aut(G)$ with respect to the natural action of $\Aut(G)$ on $G$, is denoted by $\Hol(G)$. The fixed point subgroup of an automorphism $\alpha$ of some group is denoted by $\fix(\alpha)$. For a finite group $G$, the solvable radical of $G$ (the unique largest solvable normal subgroup of $G$, see \cite[paragraph after 5.1.2, p.~122]{Rob96a}) is denoted by $\Rad(G)$, and the socle of $G$ (the subgroup generated by the minimal nontrivial normal subgroups of $G$) by $\Soc(G)$.

For a prime $p$ and $n\in\IN^+$, we denote by $\nu_p(n)$ the $p$-adic valuation of $n$, i.e., the largest $k\in\IN$ such that $p^k\mid n$. For a prime power $q$, the finite field with $q$ elements is denoted by $\IF_q$ and its group of units by $\IF_q^{\ast}$; moreover, for $d\in\IN^+$ and a subgroup $G$ of the general linear group $\GL_d(q)$, we denote the image of a matrix $M\in G$ under the canonical projection $G\rightarrow G/\zeta G$ onto the central quotient of $G$ (the projective version of $G$) by $\overline{M}$. For a Galois field extension $L/K$, $\Gal(L/K)$ denotes its Galois group.

The rest of the notation used in this paper is either defined at some point in the text or standard.

\section{Connection with coset word maps}\label{sec2}

In this section, we consider a certain generalization of word maps and show how they relate with our problem of proving that a given word is multiplicity-bounding. The maps considered here are also associated with certain formal expressions, which we will call \emph{coset word maps}. Let us first consider the following concepts generalizing words and word maps respectively:

\begin{definition}\label{gWordDef}
Let $G$ be a group, and let $X_1,\ldots,X_d$ be variables. Denote by $F(X_1,\ldots,X_d)$ the free group generated by $X_1,\ldots,X_d$.

\begin{enumerate}
\item A \emph{(reduced) $G$-word} is an element of the free product $F(X_1,\ldots,X_d)\ast G$.
\item Let $w$ be a $G$-word and $H\leq G$. We associate to $w$ and $H$ the \emph{word map} $w_H:H^d\rightarrow G$, defined as follows: For each $h=(h_1,\ldots,h_d)\in H^d$, denote by $\varphi_h$ the unique homomorphism $F(X_1,\ldots,X_d)\rightarrow H\leq G$ mapping $X_i\mapsto h_i$ for $i=1,\ldots,d$, and denote by $\psi_h$ the unique extension of $\varphi_h$ to a homomorphism $F(X_1,\ldots,X_d)\ast G\rightarrow G$ restricting to the identity on the free factor $G$. Then $w_H(h):=\psi_h(w)$.
\end{enumerate}
\end{definition}

We now define coset words and related concepts:

\begin{definition}\label{cosetWordDef}
Let $w=w(X_1,\ldots,X_d)$ be a reduced word in $d$ distinct variables.

\begin{enumerate}
\item For a group $G$ and elements $g_1,\ldots,g_d\in G$, the \emph{coset word of $w$ w.r.t.~$g_1,\ldots,g_d$}, denoted by $w_G^{[g_1,\ldots,g_d]}$, is the $G$-word defined as the following evaluation of a word map in the \enquote{usual} sense: $w_{F(X_1,\ldots,X_d)\ast G}(X_1g_1,\ldots,X_dg_d)$.
\item For a group $G$, a subgroup $H\leq G$ and elements $g_1,\ldots,g_d\in G$, the word map (in the sense of Definition \ref{gWordDef}(2)) $(w_G^{[g_1,\ldots,g_d]})_H$, also denoted by $w_{H,G}^{[g_1,\ldots,g_d]}$, is called the \emph{$(H,G)$-coset word map of $w$ w.r.t.~$g_1,\ldots,g_d$}.
\item For a centerless group $H$, a \emph{coset word (map) on $H$} is simply understood as an $(H,\Aut(H))$-coset word (map), viewing $H$ as a subgroup of $\Aut(H)$ in the natural way.
\item For a finite group $G$, denote by $\Pi_w(G)$ the largest fiber size of the word map $w_G:G^d\rightarrow G$, and set $\pi_w(G):=\Pi_w(G)/|G|^d\in\left(0,1\right]$.
\item For finite groups $H\leq G$, denote by $\Gamma_w(H,G)$ the largest fiber size of one of the $(H,G)$-coset word maps $w_{H,G}^{[g_1,\ldots,g_d]}$, where $g_1,\ldots,g_d$ run over $G$. Moreover, set $\gamma_w(H,G):=\Gamma_w(H,G)/|H|^d\in\left(0,1\right]$.
\item For a finite centerless group $H$, set $\Gamma_w(H):=\Gamma_w(H,\Aut(H))$ and $\gamma_w(H):=\gamma_w(H,\Aut(H))$.
\end{enumerate}
\end{definition}

Note that by definition, $w_{H,G}^{[g_1,\ldots,g_d]}(h_1,\ldots,h_d)=w(h_1g_1,\ldots,h_dg_d)\in G$. With a standard coset-wise counting argument (bounding both the number of cosets of $N^d$ in $G^d$ ($d$ the number of distinct variables in the word $w$) that intersect with some fixed fiber of $w_G$ and the maximum intersection size of such a coset with the fiber), one can show:

\begin{lemma}\label{cosetWiseLem}
Let $w$ be a reduced word, $G$ a finite group and $N\unlhd G$. Then $\Pi_w(G)\leq\Pi_w(G/N)\cdot\Gamma_w(N,G)$, or equivalently, $\pi_w(G)\leq\pi_w(G/N)\cdot\gamma_w(N,G)$.\qed
\end{lemma}

The following concepts will be helpful in our subsequent study of coset word maps on nonabelian finite characteristically simple groups:

\begin{definition}\label{variationDef}
Let $w=w(X_1,\ldots,X_d)$ be a reduced word in $d$ distinct variables.

\begin{enumerate}
\item For any variable $X$ (not necessarily among the $X_i$), the \emph{multiplicity of $X$ in $w$}, denoted $\mu_w(X)$, is the number of occurrences of $X^{\pm1}$ in $w$.
\item A \emph{variation of $w$} is a word obtained from $w$ by adding, for $i=1,\ldots,d$, to each occurrence of $X_i^{\pm1}$ in $w$ a second index from the range $\{1,\ldots,\mu_w(X_i)\}$.
\item For groups $H\leq G$, a \emph{varied $(H,G)$-coset word (map) of $w$} is just an $(H,G)$-coset word (map) of one of the variations of $w$, and for a centerless group $H$, a \emph{varied coset word (map) of $w$ on $H$} is a varied $(H,\Aut(H))$-coset word (map) of $w$.
\item For any group $G$, we call two $G$-words $w_1(Y_1,\ldots,Y_d)$ and $w_2(Z_1,\ldots,Z_d)$ in the same number of variables \emph{equivalent} if and only if there exists a permutation $\sigma\in\Sym_d$ such that $w_1(Z_{\sigma(1)},\ldots,Z_{\sigma(d)})=w_2(Z_1,\ldots,Z_d)$.
\end{enumerate}
\end{definition}

For example, $X_{1,2}X_{2,2}X_{1,1}^{-1}X_{2,2}^{-1}$ is a variation of the commutator word $[X_1,X_2]=X_1X_2X_1^{-1}X_2^{-1}$, but $X_{1,3}X_{2,2}X_{1,1}^{-1}X_{2,2}^{-1}$ is not, because $3>2=\mu_{[X_1,X_2]}(X_1)$.

\begin{remark}\label{variationRem}
We note the following basic facts about variations of a given reduced word $w$ in $d$ distinct variables and of length $l$:

\begin{enumerate}
\item All variations of $w$ are still reduced words of length $l$.
\item $w$ only has finitely many variations; more precisely, the number of variations of $w$ is just $\prod_{k=1}^d{\mu_w(X_k)^{\mu_w(X_k)}}$.
\end{enumerate}
\end{remark}

The following is clear from the definition:

\begin{lemma}\label{equivalentLem}
Let $H\leq G$ be groups and $w_1,w_2$ equivalent $G$-words. Then for some coordinate permutation $\sigma$ on $H^d$, we have that $(w_1)_H\circ\sigma=(w_2)_H$. In particular, $(w_1)_H:H^d\rightarrow G$ is constant if and only if $(w_2)_H$ is constant.\qed
\end{lemma}

We now turn to studying coset word maps on nonabelian finite characteristically simple groups, starting with the following basic lemma:

\begin{lemma}\label{equationSystemLem}
Let $S$ be a nonabelian finite simple group, $n\in\IN^+$, $w=w(X_1,\ldots,X_d)=x_1^{\epsilon_1}\cdots x_l^{\epsilon_l}$ a reduced word of length $l$ (in particular, $\epsilon_i\in\{\pm1\}$) in $d$ distinct variables. Denote by $\iota$ the unique function $\{1,\ldots,l\}\rightarrow\{1,\ldots,d\}$ such that for $i=1,\ldots,l$, $x_i=X_{\iota(i)}$.

Fix automorphisms $\vec{\alpha_1},\ldots,\vec{\alpha_d}\in\Aut(S^n)=\Aut(S)\wr\Sym_n$ (see \cite[3.3.20, p.~90]{Rob96a}), writing $\vec{\alpha_k}=(\alpha_{k,1},\ldots,\alpha_{k,n})\circ\sigma_i$, where each $\alpha_{k,j}$ is an automorphism of $S$ and $\sigma$ is a coordinate permutation on $S^n$, naturally associated with a permutation from $\Sym_n$ which, by abuse of notation, we also call $\sigma$.

Finally, for $j=1,\ldots,l$, define $\chi_j:=\begin{cases}\sigma_{\iota(1)}^{\epsilon_1}\cdots\sigma_{\iota(j-1)}^{\epsilon_{j-1}}, & \text{if }\epsilon_j=1, \\ \sigma_{\iota(1)}^{\epsilon_1}\cdots\sigma_{\iota(j)}^{\epsilon_j}, & \text{if }\epsilon_j=-1\end{cases}\in\Sym_n$.

Then for any automorphism $\vec{\beta}=(\beta_1\times\cdots\times\beta_n)\circ\psi\in\Aut(S^n)$, a necessary condition that the fiber of $\vec{\beta}$ under the coset word map $w_{S^n}^{[\vec{\alpha_1},\ldots,\vec{\alpha_n}]}$ is nonempty is that $w_{\Sym_n}(\sigma_1,\ldots,\sigma_d)=\psi$. If this necessary condition is met, then that fiber consists precisely of those $((s_{1,1},\ldots,s_{1,n}),\ldots,(s_{d,1},\ldots,s_{d,n}))\in(S^n)^d$ which are solutions of the equation system, consisting of $n$ equations with variables $s_{k,i}$ ranging over $S$, whose $i$-th equation, for $i=1,\ldots,n$ is the following:

\begin{equation}\label{coordinateEq}
(s_{\iota(1),\chi_1^{-1}(i)}\alpha_{\iota(1),\chi_1^{-1}(i)})^{\epsilon_1}\cdots(s_{\iota(l),\chi_l^{-1}(i)}\alpha_{\iota(l),\chi_l^{-1}(i)})^{\epsilon_l}=\beta_i.
\end{equation}

Moreover, the left-hand side of Equation (\ref{coordinateEq}) is an $\Aut(S)$-word equivalent to a varied coset word of $w$ on $S$.
\end{lemma}

\begin{proof}
For $k=1,\ldots,d$, set $\vec{s_k}:=(s_{k,1},\ldots,s_{k,n})$ and $\vec{t_k}:=(s_{k,1}\alpha_{k,1},\ldots,s_{k,n}\alpha_{k,n})$. That $(\vec{s_1},\ldots,\vec{s_n})$ lies in the fiber of $\vec{\beta}$ under $w_{S^n}^{[\vec{\alpha_1},\ldots,\vec{\alpha_n}]}$ by definition just means that

\[
(\vec{t_{\iota(1)}}\sigma_{\iota(1)})^{\epsilon_1}\cdots(\vec{t_{\iota(l)}}\sigma_{\iota(l)})^{\epsilon_l}=(\beta_1,\ldots,\beta_n)\psi.
\]

Transform the left-hand side of this equation by moving all the $\sigma_{\iota(j)}^{\epsilon_j}$ to the right, replacing the factors $(\vec{t_{\iota(j)}})^{\epsilon_j}$ by the corresponding conjugates (obtained by coordinate permutations). The case distinction in the definition of $\chi_j$ comes from the observation that if $\epsilon_j=1$, then $(\vec{t_{\iota(j)}}\sigma_{\iota(j)})^{\epsilon_j}=\vec{t_{\iota(j)}}\sigma_{\iota(j)}$, so that $\sigma_{\iota(j)}$ is already to the right of $\vec{t_{\iota(j)}}$ and hence does not affect it, whereas if $\epsilon_j=-1$, then $(\vec{t_{\iota(j)}}\sigma_{\iota(j)})^{\epsilon_j}=\sigma_{\iota(j)}^{-1}\vec{t_{\iota(j)}}^{-1}$ so that $\vec{t_{\iota(j)}}^{-1}$ is affected when moving $\sigma_{\iota(j)}^{-1}$ to the right. These transformations result in the following equivalent equation:

\[
(\vec{t_{\iota(1)}}^{\epsilon_1})^{\chi_1}\cdots(\vec{t_{\iota(l)}}^{\epsilon_l})^{\chi_l}w(\sigma_1,\ldots,\sigma_d)=(\beta_1,\ldots,\beta_n)\psi
\]

Considering this equation in $\Aut(S^n)$ modulo the normal subgroup $\Aut(S)^n$, we derive the necessity of $w(\sigma_1,\ldots,\sigma_d)=\psi$, and if this holds, we can cancel $\psi$ on both sides and derive the described system of equations by comparing, for $i=1,\ldots,n$, the $i$-th coordinates of the tuples on both sides.

For the last assertion, just replace the second indices of the $s_{k,i}$ appearing in the left-hand side of Equation (\ref{coordinateEq}) by natural numbers from $\{1,\ldots,\mu_w(s_k)\}$ in an injective way (which is possible since for fixed $k$, the $s_{k,i}$ occur precisely at those positions in the left-hand side where $s_k$ occurs in $w$.
\end{proof}

We now introduce some concepts establishing a connection with multiplicity-bounding words:

\begin{definition}\label{stronglyDef}
Let $w=w(X_1,\ldots,X_d)$ be a reduced word in $d$ distinct variables.

\begin{enumerate}
\item We say that $w$ is \emph{strongly multiplicity-bounding} if and only if for all nonabelian finite simple groups $S$, the following holds: If $f$ is a varied coset word map of $w$ on $S$, associated with a varied coset word $v$ of $w$ on $S$ such that $v$ is equivalent to the LHS of one of the $n$ equations described in Lemma \ref{equationSystemLem}, for suitable choices of $n$ and $\vec{\beta}$ with $w(\sigma_1,\ldots,\sigma_d)=\psi$, then $f$ is not constant.
\item We say that $w$ is \emph{very strongly multiplicity-bounding} if and only if for all nonabelian finite simple groups $S$, none of the varied coset word maps of $w$ on $S$ is constant.
\item We say that $w$ is \emph{weakly multiplicity-bounding} if and only if for all nonabelian finite simple groups $S$, none of the coset word maps of $w$ on $S$ is constant.
\end{enumerate}
\end{definition}

The concept of a weakly multiplicity-bounding word will be used in Sections \ref{sec5} and \ref{sec6}. We have the following \enquote{hierarchy} between these word properties, justifying the chosen terminology:

\begin{proposition}\label{stronglyProp}
The following hold:

\begin{enumerate}
\item Very strongly multiplicity-bounding words are strongly multiplicity-bounding.
\item Strongly multiplicity-bounding words are multiplicity-bounding.
\item Multiplicity-bounding words are weakly multiplicity-bounding.
\end{enumerate}
\end{proposition}

In particular, one can use the study of coset word maps on nonabelian finite simple groups to show that a given word is multiplicity-bounding. For the proof of Proposition \ref{stronglyProp}(2), we associate with every finite group some sequences of finite groups, defined as follows:

\begin{definition}\label{sequencesDef}
Let $G$ be a finite group. We recursively define a sequence $(G_n)_{n\in\IN}$ of finite groups as follows: $G_0:=G$ and $G_{n+1}:=(G_n/\Rad(G_n))/\Soc(G_n/\Rad(G_n))$. Moreover, for $n\in\IN^+$, we set:

\begin{enumerate}
\item $\R_n(G):=\Rad(G_n)$,
\item $\C_n(G):=\Soc(G_n/\Rad(G_n))$.
\end{enumerate}
\end{definition}

By this definition, it is easy to check that all the groups $\R_n(G)$ are solvable, all the groups $\C_n(G)$ are centerless completely reducible (i.e., direct products of nonabelian (finite) simple groups; see also \cite[3.3.18, p.~89]{Rob96a}) and that there exists $N=N(G)\in\IN$ such that $G_n=\{1\}$ (and thus also $\R_n(G)=\C_n(G)=\{1\}$) for all $n\geq N$. Moreover, every finite group $G$ has a characteristic series whose factors from bottom to top are, in this order, $\R_1(G),\C_1(G),\R_2(G),\C_2(G),\ldots$, continued until one of them becomes trivial for the first time. An application of Lemma \ref{cosetWiseLem} thus yields

\begin{lemma}\label{gammaLem}
Let $G$ be a finite group, $w$ a reduced word. Then $\pi_w(G)\leq\prod_{n=1}^{\infty}{\gamma_w(\C_n(G))}$.\qed
\end{lemma}

The following lemma, similar in spirit to \cite[Lemma 4.4]{Bor17b}, will also be used in the proof of Proposition \ref{stronglyProp}(2):

\begin{lemma}\label{nonUnivLem}
With assumptions and notation as in Lemma \ref{equationSystemLem}, set $\epsilon(w,S):=1-1/|S|^l$, and assume that none of the $n$ equations from the system of equations discussed is universally solvable over $S$. Then the fiber of $\vec{\beta}$ under $w_{S^n}^{[\vec{\alpha_1},\ldots,\vec{\alpha_n}]}$ has size at most $\epsilon(w,S)^{\lceil n/l^2\rceil}|S|^{nd}\leq\epsilon(w,S)|S|^{nd}$.
\end{lemma}

\begin{proof}
The proof idea is the same as for \cite[Lemma 4.4]{Bor17b}, but for the reader's convenience, we give the proof here separately.

If $m_i$ denotes the number of variables occurring in the $i$-th equation, then by assumption, that equation has at most $|S|^{m_i}-1=(1-1/|S|^{m_i})|S|^{m_i}\leq\epsilon(w,S)|S|^{m_i}$ many solutions. Our goal is to find indices $i_1,\ldots,i_{\lceil n/l^2\rceil}$ such that the corresponding equations from the system have pairwise disjoint variable sets. Indeed, once this is accomplished, it follows that the projection of the fiber of $\vec{\beta}$ under $w_{S^n}^{[\vec{\alpha_1},\ldots,\vec{\alpha_n}]}$ onto the coordinates corresponding to variables appearing in one of those $\lceil n/l^2\rceil$ many equations has size at most $\prod_{t=1}^{\lceil n/l^2\rceil}{\epsilon(w,S)|S|^{m_{i_t}}}$, hence proportion at most $\epsilon(w,S)^{\lceil n/l^2\rceil}$ within $S^{m_{i_1}+\cdots+m_{i_{\lceil n/l^2\rceil}}}$, and this is also an upper bound on the proportion of the fiber itself within $(|S|^n)^d$.

For finding these indices, we proceed by recursion. The first index, $i_1$, can be chosen arbitrarily from $\{1,\ldots,n\}$. Suppose now that we have already found indices $i_1,\ldots,i_t$ such that the corresponding equations have pairwise disjoint sets of second variable indices (which is sufficient for the variable sets to be pairwise disjoint), say $I_1,\ldots,I_t$, and note that all these sets are of cardinality at most $l$. It is sufficient to choose, as the value of $i_{t+1}$, any index $i$ not from $\bigcup_{r=1}^t\bigcup_{j=1}^l{\chi_j^{-1}[I_r]}$, a set of cardinality at most $tl^2$. Hence as long as $n>tl^2$, i.e., as long as $\lceil n/l^2\rceil\geq t+1$, we can choose a suitable index $i_{t+1}$. This concludes the proof.
\end{proof}

\begin{proof}[Proof of Proposition \ref{stronglyProp}]
For (1): This is clear by definition.

For (2): Fix $\rho\in\left(0,1\right]$, a finite group $G$ with $\pi_w(G)\geq\rho$ and a nonabelian finite simple group $S$. We need to show that the multiplicity of $S$ as a composition factor of $G$ is bounded from above in terms of $\rho$, $w$ and $S$. By the remarks after Definition \ref{sequencesDef}, it is clear that this multiplicity is the sum of the multiplicities of $S$ as a composition factor in the groups $\C_n(G)$, $n\in\IN^+$.

Denote by $N(G,S)$ the number of $n\in\IN^+$ such that $\C_n(G)$ has $S$ as a composition factor. We first show that $N(G,S)$ is bounded from above in terms of $\rho$, $w$ and $S$. To this end, let $R$ be any finite centerless completely reducible group having $S$ as a composition factor. Then we can write $R=\prod_{s=1}^r{S_s^{m_s}}$, where the $S_s$ are nonabelian finite simple groups, the $m_s\in\IN^+$, $r\in\IN^+$ and w.l.o.g.~$S_1=S$. By \cite[3.3.20, p.~90]{Rob96a}, every automorphism of $R$ splits as a product of automorphisms on the single factors $S_s^{m_s}$, from which it follows that $\gamma_w(R)\leq\prod_{s=1}^{r}{\gamma_w(S_s^{m_s})}\leq\gamma_w(S^{m_1})$. In view of Lemmata \ref{equationSystemLem} and \ref{nonUnivLem} and our assumption that $w$ is strongly multiplicity-bounding, it follows that $\gamma_w(R)\leq\epsilon(w,S)$, and so $\rho\leq\pi_w(G)\leq\epsilon(w,S)^{N(G,S)}$ by Lemma \ref{gammaLem}, which proves that indeed, $N(G,S)$ is bounded as asserted.

To complete the proof, it remains to show that for fixed $n\in\IN^+$ such that $S$ occurs as a composition factor in $\C_n(G)$, the multiplicity $m_n$ of $S$ in $\C_n(G)$ is bounded in terms of $\rho$, $w$ and $S$. By Lemma \ref{gammaLem}, the observations from the last paragraph and Lemma \ref{nonUnivLem}, we have $\rho\leq\gamma_w(\C_n(G))\leq\gamma_w(S^{m_n})\leq\epsilon(w,S)^{\lceil m_n/l^2\rceil}$, proving that $m_n$ is indeed bounded from above as required.

For (3): We show the contraposition. Assume that for some reduced word $w=w(X_1,\ldots,X_d)$, some nonabelian finite simple group $S$ and some $\alpha_1,\ldots,\alpha_d\in\Aut(S)$, $w(s_1\alpha_1,\ldots,s_d\alpha_d)=w(\alpha_1,\ldots,\alpha_d)$ for all $s_1,\ldots,s_d\in S$. For $n\in\IN^+$, denote by $G_n$ the subgroup of $\Aut(S)^n$ generated by $S^n$ and the \enquote{diagonal subgroup} $\{(\alpha,\ldots,\alpha)\mid\alpha\in\Aut(S)\}$; for brevity, for $\alpha\in\Aut(S)$, denote by $\alpha^{(n)}$ the $n$-tuple $(\alpha,\ldots,\alpha)$. Then clearly, $w(\vec{s_1}\alpha_1^{(n)},\ldots,\vec{s_d}\alpha_d^{(n)})=w(\alpha_1,\ldots,\alpha_d)^{(n)}=w(\alpha_1^{(n)},\ldots,\alpha_d^{(n)})$ for all $\vec{s_1},\ldots,\vec{s_d}\in S^n$, so that the word map $w_{G_n}$ has a fiber of size at least $|S|^n=\frac{1}{|\Out(S)|}|G_n|$. Since this construction works for all $n\in\IN^+$, we get that it is \emph{not} possible to bound the multiplicity of $S$ as a composition factor in a finite group $G$ such that $w_G$ has a fiber of size at least $\frac{1}{|\Out(S)|}|G|$. Hence $w$ is not multiplicity-bounding, as we wanted to show.
\end{proof}

\section{Some simple examples of very strongly multiplicity-bounding words}\label{sec3}

The following proposition allows us to give some examples of words that are very strongly multiplicity-bounding:

\begin{proposition}\label{examplesProp}
The following reduced words are very strongly multiplicity-bounding:

\begin{enumerate}
\item Words in which some variable occurs with multiplicity $1$.
\item Words of the form $w_1(X_1,\ldots,X_{d-1})X_dw_2(X_1,\ldots,X_{d-1})X_dw_3(X_1,\ldots,X_{d-1})$, where $w_1,w_2,w_3$ are reduced.
\item Words of the form $w_1(X_1,\ldots,X_{d-1})X_d^{\pm1}w_2(X_1,\ldots,X_{d-1})X_d^{\mp1}w_3(X_1,\ldots,X_{d-1})$, where $w_1,w_2,w_3$ are reduced and $w_2$ is very strongly multiplicity-bounding.
\end{enumerate}
\end{proposition}

For the proofs of Proposition \ref{examplesProp}(2,3), we will use information on the fiber structure of certain families of functions on groups. The first result which we will use is \cite[Proposition 5.1.1]{Bor17a}, whose formulation we give here for the reader's convenience. It uses the notation $\P_{-1}(\alpha)$ for an automorphism $\alpha$ of some group $G$, which denotes the set of elements $g\in G$ that $\alpha(g)=g^{-1}$:

\begin{proposition}\label{fiberProp1}
Let $G$ be a group, $\alpha$ an automorphism of $G$, and fix $c\in G$. Consider the map $\f_{c,\alpha}:G\rightarrow G,g\mapsto gc\alpha(g)$. Then for $g_1,g_2\in G$, we have $\f_{c,\alpha}(g_1)=\f_{c,\alpha}(g_2)$ if and only if $g_2\in\P_{-1}(\tau_{g_1}\circ\tau_c\circ\alpha)g_1$. In other words, the fibers of $\f_{c,\alpha}$ are just the subsets of $G$ of the form $\P_{-1}(\tau_g\circ\tau_c\circ\alpha)g$, $g\in G$.\qed
\end{proposition}

The second result of this nature which we will use is the following:

\begin{proposition}\label{fiberProp2}
Let $G$ be a group, $\alpha$ an automorphism of $G$, and fix $c\in G$. Consider the map $\h_{c,\alpha}:G\rightarrow G,g\mapsto g^{-1}c\alpha(g)$. Then for $g_1,g_2\in G$, we have $\h_{c,\alpha}(g_1)=\h_{c,\alpha}(g_2)$ if and only if $g_2\in\fix(\tau_c\circ\alpha)g_1$. In other words, the fibers of $\h_{c,\alpha}$ are just the right cosets of $\fix(\tau_c\circ\alpha)$ in $G$.
\end{proposition}

For $c=1$, this yields the well-known fact that the fibers of a map $G\rightarrow G$ of the form $g\mapsto g^{-1}\alpha(g)$, $\alpha$ an automorphism of $G$ are just the right cosets of $\fix(\alpha)$ in $G$.

\begin{proof}[Proof of Proposition \ref{fiberProp2}]
By a simple equivalence transformation, we have that $g_1^{-1}c\alpha(g_1)=g_2^{-1}c\alpha(g_2)$ if and only if $g_2g_1^{-1}=c\alpha(g_2g_1^{-1})c^{-1}=(\tau_c\circ\alpha)(g_2g_1^{-1})$, which proves the assertion.
\end{proof}

\begin{proof}[Proof of Proposition \ref{examplesProp}]
For (1): Let $w=w(X_1,\ldots,X_d)$ be a reduced word in $d$ distinct variables, and w.l.o.g., assume that $X_1$ occurs with multiplicity $1$ in $w$. Let $w'=w'(X_{1,1},X_{k_2,m_2},\ldots,X_{k_e,m_e})$ be any variation of $w$ containing precisely $e$ distinct variables, and note that $w'$ contains the variable $X_{1,1}$ with multiplicity $1$. Finally, let $S$ be any nonabelian finite simple group and $f:S^e\rightarrow\Aut(S)$ some coset word map of $w'$ on $S$. Then for any fixed $\beta\in\Aut(S)$ and $s_2,\ldots,s_e\in S$, in the equation $f(x,s_2,\ldots,s_e)=\beta$, the variable $x$ occurs precisely once as a factor in the product on the left-hand side and thus can be isolated, showing that the equation has at most one solution $x\in S$. But $|S|>1$, so the equation cannot be universally solvable, as required.

For (2): Let $w'=w'(X_{k_1,m_1},\ldots,X_{k_{e-1},m_{e-1}},X_{d,1})$ be a variation of $w$ in precisely $e$ distinct variables where w.l.o.g., both occurrences of $X_d$ in $w$ get the second index $1$ (if they get different second indices, then $X_{d,1}$ occurs in $w'$ with multiplicity $1$, and we can conclude as in the argument for point (1)). Let $S$ be any nonabelian finite simple group and $f:S^e\rightarrow\Aut(S)$ some coset word map of $w'$ on $S$. For any fixed $\beta\in\Aut(S)$ and $s_1,\ldots,s_{e-1}\in S$, consider the univariate equation $f(s_1,\ldots,s_{e-1},x)=\beta$, whose left-hand side is some $\Aut(S)$-word in the single variable $x$ of the form $\alpha_1x\alpha_2x\alpha_3$ with $\alpha_i\in\Aut(S)$ fixed. Hence the equation is equivalent to $x\alpha_2x=\alpha_1^{-1}\beta\alpha_3^{-1}$, and by Proposition \ref{fiberProp1}, if this equality holds for all $x\in S$, then the automorphism $\alpha_2$ of $S$ maps all elements of $S$ to their inverses. In other words, the inversion map on $S$ is an automorphism of $S$ then, but this is only possible if $S$ is abelian, a contradiction.

For (3): Since a reduced word is strongly multiplicity-bounding if and only if its inverse is (see also Lemma \ref{permSubLem}), we can assume w.l.o.g. that the first occurrence of $X_d$ in $w$ has exponent $-1$ (and the second exponent $1$). Again, let $w'=w'(X_{k_1,m_1},\ldots,X_{k_{e-1},m_{e-1}},X_{d,1})$ be a variation of $w$ in precisely $e$ distinct variables where both occurrences of $X_d^{\pm1}$ in $w$ got the second index $1$. Let $S$ be any nonabelian finite simple group and $f:S^e\rightarrow\Aut(S)$ some coset word map of $w'$ on $S$. By the assumption that the segment $w_2(X_1,\ldots,X_{d-1})$ of $w$ is strongly multiplicity-bounding, we can fix $s_1,\ldots,s_{e-1}\in S$ such that the univariate $\Aut(S)$-word $f(s_1,\ldots,s_{e-1},x)$ is of the form $\alpha_1x^{-1}\alpha_2x\alpha_3$ with $\alpha_i\in\Aut(S)$ and $\alpha_2\not=\id$. Fix any $\beta\in\Aut(S)$. Then as before, the equation $f(s_1,\ldots,s_{e-1},x)=\beta$ is equivalent to $x^{-1}\alpha_2x=\alpha_1^{-1}\beta\alpha_3^{-1}$, and by Proposition \ref{fiberProp2}, if this holds true for all $x\in S$, then the automorphism $\alpha_2$ of $S$ fixes all elements of $S$, a contradiction.
\end{proof}

Using Proposition \ref{examplesProp}, one can show the following stronger form of Theorem \ref{mainTheo}(2):

\begin{corollary}\label{examplesCor}
The words $\gamma_n$ defined in Theorem \ref{mainTheo}(2) are all very strongly multiplicity-bounding.
\end{corollary}

\begin{proof}
Proceed by induction on $n$. For $n=1$, this is clear by Proposition \ref{examplesProp}(1), and for the induction step, use Proposition \ref{examplesProp}(3).
\end{proof}

\section{Reduction arguments}\label{sec4}

The main result of this section is Proposition \ref{mainReductionProp} below, which reduces proving that a given reduced word $w$ is very strongly multiplicity-bounding to a finite problem. This builds on results and techniques from Nikolov's recent paper \cite{Nik16a}.

First, we recall the following notion and notation from \cite[Definition 1.2.1]{Bor17b} and note its connection with coset word maps:

\begin{definition}\label{automorphicWordDef}
Let $G$ be a group, $w=w(X_1,\ldots,X_d)$ a reduced word of length $l$ in $d$ distinct variables. Write $w=x_1^{\epsilon_1}\cdots x_l^{\epsilon_l}$, where each $x_j\in\{X_1,\ldots,X_d\}$ and each $\epsilon_j\in\{\pm1\}$, and denote by $\iota$ the unique function $\{1,\ldots,l\}\rightarrow\{1,\ldots,d\}$ such that $x_j=X_{\iota(j)}$ for $j=1,\ldots,l$. Finally, let $\alpha_1,\ldots,\alpha_l$ be automorphisms of $G$. Then the \emph{automorphic word map of $w$ on $G$ with respect to $\alpha_1,\ldots,\alpha_l$}, denoted by $w_G^{(\alpha_1,\ldots,\alpha_l)}$, is the map $G^d\rightarrow G$ sending $(g_1,\ldots,g_d)\mapsto\alpha_1(g_{\iota(1)})^{\epsilon_1}\cdots\alpha_l(g_{\iota(l)})^{\epsilon_l}$.
\end{definition}

\begin{lemma}\label{automorphicCosetLem}
Let $G\unlhd O$ be groups, $w=w(X_1,\ldots,X_d)$ a reduced word of length $l$ in $d$ distinct variables, $x_1,\ldots,x_l$, $\epsilon_1,\ldots,\epsilon_l$ and $\iota$ as in Definition \ref{automorphicWordDef}. Moreover, let $o_1,\ldots,o_d\in O$, denote by $\varphi:O\rightarrow\Aut(G)$ the conjugation action of $O$ on $G$ and set, for $j=1,\ldots,l$, $\omega_j:=\begin{cases}\varphi(o_{\iota(1)}^{\epsilon_1}\cdots o_{\iota(j-1)}^{\epsilon_{j-1}}), & \text{if }\epsilon_j=1, \\ \varphi(o_{\iota(1)}^{\epsilon_1}\cdots o_{\iota(j)}^{\epsilon_j}), & \text{if }\epsilon_j=-1\end{cases}$. Then for all $g_1,\ldots,g_d$, we have the following equality relating the $(G,O)$-coset word map $w_{G,O}^{[o_1,\ldots,o_d]}:G^d\rightarrow O$, the automorphic word map $w_G^{(\omega_1,\ldots,\omega_l)}:G^d\rightarrow G$ and the word map $w_O:O^d\rightarrow O$:

\[
w_{G,O}^{[o_1,\ldots,o_d]}(g_1,\ldots,g_d)=w_G^{(\omega_1,\ldots,\omega_l)}(g_1,\ldots,g_d)\cdot w_O(o_1,\ldots,o_d).
\]

In particular, there is a cardinality-preserving bijection between the fibers of $w_{G,O}^{[o_1,\ldots,o_d]}$ and the fibers of $w_G^{(\omega_1,\ldots,\omega_l)}$.
\end{lemma}

\begin{proof}
This is very similar to the proof of Lemma \ref{equationSystemLem}: It follows by moving, in the evaluation $w_{G,O}^{[o_1,\ldots,o_d]}$, all occurrences of the $o_k^{\pm1}$ to the right, replacing the $g_k^{\pm1}$ by appropriate $O$-conjugates. The \enquote{In particular} follows since $w_O(o_1,\ldots,o_d)$ does not depend on the $g_k$.
\end{proof}

We want techniques to show that constancy of some coset word map on a nonabelian finite simple group $S$ leads to constancy of a suitable coset word map on a subgroup $U$ of $S$. In most applications, $U$ will itself be nonabelian simple, but there will also be cases where it is not even centerless. Because of this, it will be more natural to work with $(G,\Hol(G))$-coset word maps in general here.

We note that for centerless groups $G$, there is a constancy-preserving duality between coset word maps of $w$ on $G$ (i.e., $(G,\Aut(G))$-coset word maps of $w$) and $(G,\Hol(G))$-coset word maps:

\begin{lemma}\label{dualityLem}
Let $G$ be a centerless group, $w$ a reduced word in $d$ distinct variables and $\alpha_1,\ldots,\alpha_d\in\Aut(G)$. Then the coset word map $w_{G,\Aut(G)}^{[\alpha_1,\ldots,\alpha_d]}:G^d\rightarrow\Aut(G)$ is constant if and only if the coset word map $w_{G,\Hol(G)}^{[\alpha_1,\ldots,\alpha_d]}:G^d\rightarrow\Hol(G)$ is constant.
\end{lemma}

\begin{proof}
This is clear by Lemma \ref{automorphicCosetLem}, applied with $O:=\Aut(G)$ and $O:=\Hol(G)$ respectively, and the fact that automorphisms of $G$, when viewed as elements of $\Aut(G)$ and $\Hol(G)$ respectively, act in both cases in the same way by conjugation on the canonical copy of $G$ in the respective overgroup.
\end{proof}

The following properties of subgroups will be useful for our reduction arguments:

\begin{definition}\label{sufficientlyDef}
Let $G$ be a finite centerless group and $H\leq G$.

\begin{enumerate}
\item We say that $H$ is \emph{sufficiently invariant in $G$} if and only if every coset of $G$ in $\Aut(G)$ contains an automorphism leaving $H$ invariant.
\item We say that $H$ is \emph{conjugacy-uniquely maximal in $G$} or a \emph{conjugacy-unique maximal subgroup of $G$} if and only if $H$ is maximal in $G$ and $G$ has precisely one conjugacy class of maximal subgroups isomorphic to $H$.
\end{enumerate}
\end{definition}

\begin{lemma}\label{sufficientlyLem}
Let $G$ be a finite centerless group, $H\leq G$, $w=w(X_1,\ldots,X_d)$ a reduced word in $d$ distinct variables.

\begin{enumerate}
\item If $H$ is conjugacy-uniquely maximal in $G$, then $H$ is sufficiently invariant in $G$.
\item If $\alpha_1,\ldots,\alpha_d$ are automorphisms of $G$ under which $H$ is invariant, and if the $(G,\Hol(G))$-coset word map of $w$ with respect to $\alpha_1,\ldots,\alpha_d$ is constant, then the $(H,\Hol(H))$-coset word map of $w$ with respect to the restrictions of $\alpha_1,\ldots,\alpha_d$ to $H$ is constant.
\end{enumerate}
\end{lemma}

\begin{proof}
For (1): Let $\alpha$ be any automorphism of $G$. By the assumptions, it follows that there exists $g\in G$ such that $\alpha[H]=\tau_g[H]$, whence $H$ is invariant under the automorphism $\tau_{g^{-1}}\alpha=g^{-1}\alpha$, which is from the same coset of $G$ in $\Aut(G)$ as $\alpha$.

For (2): With an argument as in the proof of Lemma \ref{dualityLem}, one can show the stronger assertion that the restriction of $w_{G,\Hol(G)}^{[\alpha_1,\ldots,\alpha_d]}$ to $H^d$ is constant if and only if $w_{H,\Hol(G)}^{[(\alpha_1)_{\mid H},\ldots,(\alpha_d)_{\mid H}]}$ is constant.
\end{proof}

The following simple observation will also be used several times:

\begin{lemma}\label{characteristicLem}
Let $G$ be a group, $N\cha G$, $w=w(X_1,\ldots,X_d)$ a reduced word in $d$ distinct variables. For an automorphism $\alpha$ of $G$, denote by $\tilde{\alpha}$ the automorphism of $G/N$ induced by $\alpha$. Assume that $\alpha_1,\ldots,\alpha_d$ are automorphisms of $G$ such that the $(G,\Hol(G))$-coset word map $w_{G,\Hol(G)}^{[\alpha_1,\ldots,\alpha_d]}$ is constant. Then:

\begin{enumerate}
\item The $(N,\Hol(N))$-coset word map $w_{N,\Hol(N)}^{[(\alpha_1)_{\mid N},\ldots,(\alpha_d)_{\mid N}]}$ is constant.
\item The $(G/N,\Hol(G/N))$-coset word map $w_{G/N,\Hol(G/N)}^{[\tilde{\alpha_1},\ldots,\tilde{\alpha_d}]}$ is constant.
\end{enumerate}
\end{lemma}

\begin{proof}
Let $l$ be the length of $w$, and let $x_1,\ldots,x_l$, $\epsilon_1,\ldots,\epsilon_l$ and $\iota$ be as in Definition \ref{automorphicWordDef}. For $j=1,\ldots,l$, set $\omega_j:=\begin{cases}\alpha_{\iota(1)}^{\epsilon_1}\cdots\alpha_{\iota(j-1)}^{\epsilon_{j-1}}, & \text{if }\epsilon_j=1, \\ \alpha_{\iota(1)}^{\epsilon_1}\cdots\alpha_{\iota(j)}^{\epsilon_j}, & \text{if }\epsilon_j=-1\end{cases}$. Then by Lemma \ref{automorphicCosetLem} and the way conjugation on $G$ by elements from $\Aut(G)$ works in $\Hol(G)$, we get that the automorphic word map $w_G^{(\omega_1,\ldots,\omega_l)}:G^d\rightarrow G$ is constant. Therefore, $(w_G^{(\omega_1,\ldots,\omega_l)})_{\mid N^d}=w_N^{((\omega_1)_{\mid N},\ldots,(\omega_l)_{\mid N})}$ is constant, which by Lemma \ref{automorphicCosetLem} implies statement (1). Moreover, for $g_1,\ldots,g_d\in G$, by applying the canonical projection $\pi:G\rightarrow G/N$ to the valid equation $w_G^{(\omega_1,\ldots,\omega_l)}(g_1,\ldots,g_d)=1_G$, we conclude that $w_{G/N}^{(\tilde{\omega_1},\ldots,\tilde{\omega_l})}(\pi(g_1),\ldots,\pi(g_d))=1_{G/N}$, so that $w_{G/N}^{(\tilde{\omega_1},\ldots,\tilde{\omega_l})}$ is constant, which yields statement (2) by another application of Lemma \ref{automorphicCosetLem}.
\end{proof}

For later reference, we also note the following simple result, which follows immediately from the definition of a coset word map:

\begin{lemma}\label{representativesLem}
Let $G\leq O$ be groups, $w=w(X_1,\ldots,X_d)$ a reduced word in $d$ distinct variables.

\begin{enumerate}
\item For all $o_1,\ldots,o_d\in O$ and $g_1,\ldots,g_d\in G$, there is a cardinality-preserving bijection between the fibers of $w_{G,O}^{[o_1,\ldots,o_d]}$ and the fibers of $w_{G,O}^{[g_1o_1,\ldots,g_dod_]}$.
\item In particular, if $R$ is a set of representatives for the cosets of $G$ in $O$, then if any $o_1,\ldots,o_d\in O$ exist such that the $(G,O)$-coset word map $w_{G,O}^{[o_1,\ldots,o_d]}$ is constant, there even exist $o_1,\ldots,o_d\in R$ with this property.\qed
\end{enumerate}
\end{lemma}

The special case $m=1$ in Nikolov's result \cite[Proposition 7]{Nik16a} means that for a given reduced word $w$, none of the coset word maps of $w$ on large enough nonabelian finite simple groups $S$ is constantly $1$. It would be great for us if one could improve this to non-constancy of coset word maps of $w$ on sufficiently large $S$ in general, and it turns out that Nikolov's very proof actually gives this stronger result:

\begin{proposition}\label{nikolovProp}
Let $w$ be a reduced word in $d$ distinct variables and of length $l$, and denote by $m$ the maximum multiplicity with which some variable occurs in $w$. Then the following statements hold:

\begin{enumerate}
\item There exists $N=N(w)\in\IN^+$ such that for all nonabelian finite simple groups $S$ of order greater than $N(w)$, none of the varied coset word maps of $w$ on $S$ is constant.
\item $w$ is weakly multiplicity-bounding if and only if some coset word map of $w$ on one of the centerless groups from $\{\PSL_2(p^L)^n\mid p\leq m,L\leq m\cdot l,n\in\{1,2\}\}\cup\{\Suz(2^{2L-1})\mid L\leq 4ml\}$ is constant.
\item It is algorithmically decidable whether a given reduced word is very strongly multiplicity-bounding.
\end{enumerate}
\end{proposition}

\begin{proof}
For statement (1), since each word only has finitely many variations, it suffices to show, for each reduced word $w$, the existence of $M=M(w)\in\IN^+$ such that for all nonabelian finite simple groups $S$ of order greater than $M(w)$, none of the (\enquote{unvaried}) coset word maps of $w$ on $S$ is constant.

By Lemma \ref{automorphicCosetLem}, if, for some $S$, some coset word map of $w$ on $S$ is constant, then also some automorphic word map of $w$ on $S$ is constant, which by \cite[Theorem 1.1.2]{Bor17b} rules out large alternating groups and simple Lie type groups of large rank as such $S$. Hence we only need to deal with Lie type groups of bounded rank. For this, we can follow the argument in \cite[proof of Proposition 7]{Nik16a}.

First, just as in \cite[paragraph from the end of p.~868 to the beginning of p.~869]{Nik16a}, we can reduce to the cases $S=\PSL_2(q)=\PSL_2(p^L)$ and $S=\Suz(2^{2L-1})$ using Lemmata \ref{sufficientlyLem}, \ref{characteristicLem} and \ref{representativesLem}. To deal with those two cases, Nikolov chose certain ansatzes for the arguments of a hypothetical constant coset word map on $S$, which allowed him to express the assumption that the word map is constantly $1_S$ by a matrix equation with polynomial coefficients. In all the three cases considered by Nikolov ($S=\PSL_2(p^L)$ with $p>l$, $S=\PSL_2(p^L)$ with $L>l^2$ and $S=\Suz(2^{2L-1})$ with $L>4l^2$), his argument showed that a certain matrix with polynomial coefficients called $U_1$ (note that we only need the case $m=1$ of Nikolov's argument) does not always assume the identity matrix as a value when the variables from its coefficients are specified in $\IF_p$, while it is clear from the chosen ansatz that $U_1$ does become the identity matrix when all its arguments are chosen to be $0$.

By the way the matrix $U_1$ is defined, it corresponds to the non-constant part $w_G^{(\omega_1,\ldots,\omega_l)}(g_1,\ldots,g_d)$ in the factorization of the coset word map value in Lemma \ref{automorphicCosetLem} (with the formal difference that Nikolov in his argument moves the coset representatives to the left, not to the right). Hence Nikolov's result that $U_1$ assumes at least two distinct values if the field parameter $q$ is large enough actually shows that under this assumption, no coset word map of $w$ on $S$ can be constant at all, not just that it cannot be constantly $1$. This concludes the proof of the first assertion.

To prove statement (2), we also make use of Lemmata \ref{sufficientlyLem}, \ref{characteristicLem} and \ref{representativesLem}. Since $\Alt_5\cong\PSL_2(4)\cong\PSL_2(5)$ and $\Alt_6\cong\PSL_2(9)$ are of Lie type, each alternating group $\Alt_n$ with $n\geq 7$ contains $\Alt_5$ as a sufficiently invariant subgroup and every sporadic finite simple group $S$ either has some sufficiently invariant subgroup containing a simple Lie type group as a characteristic subgroup or quotient or contains another sporadic simple group with this property as a sufficiently invariant subgroup (see, for example, \cite[tables in Section 4]{Wil17a}), we conclude by Lemmata \ref{sufficientlyLem} to \ref{representativesLem} that if any $S$ on which some coset word map of $w$ is constant exists, $S$ can be chosen of Lie type.

But the same reduction argument as the one from the beginning of the third paragraph of this proof shows that if any nonabelian finite simple group $S$ of Lie type on which some coset word map of $w$ is constant exists, then some coset word map of $w$ on one of the groups $\PSL_2(q)^n=\PSL_2(p^L)^n$, $n\in\{1,2\}$, or $\Suz(2^{2L-1})$ is constant. By Lemma \ref{equationSystemLem}, this implies in any case that some coset word map of some variation $w'$ of $w$ is constant on one of the simple groups of the form $\PSL_2(q)=\PSL_2(p^L)$ or $\Suz(2^{2L-1})$. By following Nikolov's argument as we did before, we can further reduce to those finitely many cases where all the relevant parameters ($p$ and $L$ for $S=\PSL_2(p^L)$, and $L$ for $S=\Suz(2^{2L-1})$) are too small to conclude that $U_1$ can evaluate to a non-identity matrix. Nikolov gives bounds of the form $p\leq l$ and $L\leq l^2$ for $S=\PSL_2(p^L)$ and $L\leq 4l^2$ for $S=\Suz(2^{2L-1})$ for this, which are based on the bound $|C_{\alpha,\beta}|\leq l$ from \cite[Proposition 30]{Nik16a}. However, it is easy to see that the stronger bound $|C_{\alpha,\beta}|\leq m$ holds, which lets us improve Nikolov's bounds to $p\leq m$ and $L\leq ml$ for $S=\PSL_2(p^L)$ and $L\leq 4lm$ for $S=\Suz(2^{2L-1})$ (note that all these bounds still work for the variation $w'$, as it has the same length as $w$ and at most the maximum variable multiplicity of $w$). This concludes the proof of statement (2), and statement (3) is clear from it (by definition, a reduced word is very strongly multiplicity-bounding if and only if all its variations are weakly multiplicity-bounding).
\end{proof}

We are now ready to formulate and prove our main reduction result, Proposition \ref{mainReductionProp}. Statement (1) of it is an immediate consequence of Proposition \ref{nikolovProp}(2), and statements (2)--(7) give possibilities for further reductions which can be used to speed a decision algorithm for the word property of being very strongly multiplicity-bounding up. Statement (8) is a combination of the previous $7$ statements and explicitly lists the steps of such a decision algorithm.

\begin{proposition}\label{mainReductionProp}
Let $w$ be a reduced word of length $l$ and such that the maximum multiplicity with which a variable occurs in $w$ is $m$. Then the following hold:

\begin{enumerate}
\item $w$ is very strongly multiplicity-bounding if and only if none of the varied coset word maps of $w$ on any of the following nonabelian finite simple groups is constant:

\begin{itemize}
\item $\PSL_2(p^L)$ for primes $p\leq m$ and positive integers $L\leq ml$ with $p^L\geq 5$.
\item $\Suz(2^{2L-1})$ for positive integers $L$ with $2\leq L\leq 4ml$.
\end{itemize}

\item If the word map $w_{\Sym_3}=w_{\PSL_2(2)}$ is non-constant, then none of the varied coset word maps of $w$ on any of the groups $\PSL_2(2^L)$, $L\geq 2$, is constant.
\item It is the case that none of the varied coset word maps of $w$ on any of the groups $\PSL_2(2^L)$, $2\leq L\leq ml$, is constant if and only if this just holds for those $L\in\{2,\ldots,ml\}$ which are primes.
\item It is the case that none of the varied coset word maps of $w$ on any of the groups $\PSL_2(3^L)$, $2\leq L\leq ml$, is constant if any only if this just holds for those $L\in\{2,\ldots,ml\}$ which are either powers of $2$ or odd primes.
\item Let $p$ be any odd prime. If none of the varied coset word maps of $w$ are constant on any finite group $G$ of the form $\PSL_2(p^L)$, $L\in\{1,\ldots,ml\}$ a (possibly trivial) power of $2$ (note the inclusion of $L=1$ even for $p=3$ here), then none of the varied coset word maps of $w$ on any of the simple groups among the $\PSL_2(p^L)$, $1\leq L\leq ml$, is constant. Moreover, for $p>3$, the reverse implication holds trivially as well.
\item If the word map $w_{\Suz(2)}$ is non-constant, then none of the varied coset word maps of $w$ on any of the groups $\Suz(2^{2L-1})$, $L\geq 2$, is constant.
\item It is the case that none of the varied coset word maps of $w$ on any of the groups $\Suz(2^{2L-1})$, $2\leq L\leq 4ml$, is constant if and only if this just holds for those $L\in\{2,\ldots,4ml\}$ where $2L-1$ is an odd prime.
\item $w$ is very strongly multiplicity-bounding if and only if the following hold:

\begin{itemize}
\item non-constancy of the word map $w_{\Sym_3}=w_{\PSL_2(2)}$ or non-constancy of the varied coset word maps of $w$ on the groups $\PSL_2(2^L)$, $2\leq L\leq ml$ and $L$ a prime,
\item non-constancy of the varied coset word maps of $w$ on all $G\in\{\PSL_2(3^L)\mid 1\leq L\leq ml,L\text{ a power of }2\}$ or non-constancy of the varied coset word maps of $w$ on all $S\in\{\PSL_2(3^L)\mid 2\leq L\leq ml,L\text{ a power of }2\text{ or an odd prime.}\}$.
\item non-constancy of the varied coset word maps of $w$ an all $G\in\{\PSL_2(p^L)\mid 5\leq p\leq m,p\text{ a prime},1\leq L\leq ml,L\text{ a power of }2\}$.
\item non-constancy of the word map $w_{\Suz(2)}$ or non-constancy of the varied coset word maps of $w$ on all $S\in\{\Suz(2^{2L-1})\mid 2\leq L\leq 4ml,L\text{ such that }2L-1\text{ is an odd prime}\}$.
\end{itemize}
\end{enumerate}
\end{proposition}

\begin{proof}
Statement (1) follows immediately from Proposition \ref{nikolovProp}(2), Lemma \ref{equationSystemLem} and the fact that a variation of a variation of $w$ is equivalent (in the sense of Definition \ref{variationDef}(4)) to some variation of $w$.

We prove the contraposition of statement (2). Note that the field automorphisms of $\PSL_2(2^L)$ form a set of representatives for the cosets of $\PSL_2(2^L)$ in $\Aut(\PSL_2(2^L))$. Hence by Lemma \ref{representativesLem} and the assumption, there exist $L\geq 2$, a variation $w'$ of $w$, say in $e$ distinct variables, and field automorphisms $\alpha_1,\ldots,\alpha_e$ of $S=\PSL_2(2^L)$ such that the coset word map $(w')_{S,\Aut(S)}^{[\alpha_1,\ldots,\alpha_e]}$ is constant. Each $\alpha_k$ fixes the canonical copy of $\PSL_2(2)$ in $S$ element-wise, whence an application of Lemma \ref{sufficientlyLem}(2) yields that the word map of $w'$ on $\PSL_2(2)$ is constant. Since $w$ can be obtained from $w'$ by suitable variable substitutions, the word map $w_{\PSL_2(2)}$ is also constant, as required.

For statement (3), note that by Dickson's classification of the conjugacy classes of maximal subgroups of $\PSL_2(q)$ (see, for example, \cite{Kin05a}), for any $L\geq 2$ and any prime divisor $r$ of $L$, $\PSL_2(2^L)$ contains a conjugacy-unique maximal subgroup of isomorphism type $\PSL_2(2^r)$. The statement is thus clear by Lemmata \ref{sufficientlyLem}(2) and \ref{representativesLem}(2).

The proofs of statements (4) and (5) are analogous to the one of statement (3), this time using that for odd primes $p$, $L\geq 2$ and any \emph{odd} prime divisor $r$ of $L$, $\PSL_2(2^L)$ contains a conjugacy-unique maximal subgroup of isomorphism type $\PSL_2(2^{L/r})$ (see \cite[Theorem 2.1(o)]{Kin05a}).

The proof of statement (6) is analogous to the proof of statement (2); just replace $\PSL_2(2^L)$ by $\Suz(2^{2L-1})$ and $\PSL_2(2)$ by $\Suz(2)$.

To prove statement (7), use that for every odd divisor $O$ of $2L-1$, $\Suz(2^{2L-1})$ contains a canonical copy of $\Suz(2^O)$ invariant under all field automorphisms of $\Suz(2^{2L-1})$ and apply Lemmata \ref{sufficientlyLem}(2) and \ref{representativesLem}(2).

Statement (8), finally, follows from statements (1)--(7).
\end{proof}

\section{Power words}\label{sec5}

In this section, we study the special case of power words, i.e., words of the form $x^e$ for a variable $x$ and some integer $e$. For them, the situation is nicer than in general, since by Theorem \ref{powerWordTheo} below, we do not need to consider varied coset word maps of a power word when investigating whether that word is multiplicity-bounding.

\begin{theorem}\label{powerWordTheo}
Let $e\in\IZ$. The following are equivalent:

\begin{enumerate}
\item The power word $x^e$ is multiplicity-bounding.
\item The power word $x^e$ is weakly multiplicity-bounding.
\item For all $S\in\{\PSL_2(p^L)\mid p\leq e, L\leq e^2, p^L\geq 5\}\cup\{\Suz(2^{2L-1})\mid 2\leq L\leq 4e^2\}$: The $e$-th power map is non-constant on every coset of $S$ in $\Aut(S)$.
\end{enumerate}

In particular, $x^e$ is multiplicity-bounding if and only if $x^{-e}$ is multiplicity-bounding, and there is an algorithm which on input $e$ decides whether $x^e$ is multiplicity-bounding.
\end{theorem}

Note that while we have, by Proposition \ref{nikolovProp}(3), an algorithm to decide whether a given reduced word is very strongly multiplicity-bounding, it is open whether there exists a general decision algorithm for the word property of being multiplicity-bounding (see also Question \ref{algorithmQues}).

\begin{proof}[Proof of Theorem \ref{powerWordTheo}]
For \enquote{(1) $\Rightarrow$ (2)}: Clear by Proposition \ref{stronglyProp}(3).

For \enquote{(2) $\Rightarrow$ (1)}: It is easy to see that a reduced word $w$ is (weakly) multiplicity-bounding if and only if its inverse $w^{-1}$ has this property, so we may assume w.l.o.g.~that $e>0$. Note that the assumption implies the stronger statement that for each divisor $t$ of $e$, the $t$-th power map is never constant on a coset of a nonabelian finite simple group in its automorphism group. We will use this to show that $w$ is strongly multiplicity-bounding. Consider the system of equations described in Lemma \ref{equationSystemLem} for the special case $w=X_1^e$. We only need to consider the case where $\psi_1^e=\id$, i.e., where each cycle length of $\psi_1$ divides $e$. Fix $i\in\{1,\ldots,n\}$. If $d$ denotes the cycle length of $i$ under $\psi_1$, then the $i$-th equation in the system, Equation (\ref{coordinateEq}), becomes

\[
(s_{1,i}\alpha_{1,i}\cdot s_{1,\psi_1^{-1}(i)}\alpha_{1,\psi_1^{-1}(i)}\cdots s_{1,\psi_1^{-(d-1)}(i)}\alpha_{1,\psi_1^{-(d-1)}(i)})^{e/d}=\beta_i,
\]

where the variables $s_{1,i},s_{1,\psi_1^{-1}(i)},\ldots,s_{1,\psi_1^{-(d-1)}(i)}$ appearing in the left-hand side are pairwise distinct. Setting the variables other than $s_{1,i}$ equal to $1_S$ and letting $s_{1,i}$ run through $S$, the left-hand side of the $i$-th equation runs through all $(e/d)$-th powers of the elements of an entire coset of $S$ in $\Aut(S)$. Hence if the $i$-th equation of the system was universally solvable, it would follow that the $(e/d)$-th power map is constant on the corresponding coset, contradicting our assumption. Therefore, none of the equations is universally solvable, so $w$ is strongly multiplicity-bounding, and we can conclude by an application of Proposition \ref{stronglyProp}(2).

For \enquote{(2) $\Rightarrow$ (3)}: Clear by definition.

For \enquote{(3) $\Rightarrow$ (2)}: Show the contraposition, and assume $e>0$ w.l.o.g.~again. If $x^e$ is not weakly multiplicity-bounding, then by Proposition \ref{nikolovProp}(2), the $e$-th power map is constant on some coset of $G$ in $\Aut(G)$ for $G$ one of the finite centerless groups $\PSL_2(p^L)^n,p\leq e,L\leq e^2,p^L\geq 5,n\in\{1,2\}$ or $\Suz(2^{2L-1}),2\leq L\leq 4e^2$. If $G$ is not isomorphic to some $\PSL_2(p^L)^2$, we are clearly done, so assume that the $e$-th power map is constant on some coset $S^2\alpha$ of $S^2$ in $\Aut(S^2)=\Aut(S)\wr\Sym_2$ for $S=\PSL_2(p^L)$ with $p$ and $L$ in the above described ranges. Write $\alpha=(\alpha_1,\alpha_2)\sigma$ with $\alpha_1,\alpha_2\in\Aut(S)$ and $\sigma\in\Sym_2=\{\id,(1,2)\}$. If $\sigma=\id$, by projecting onto the first coordinate, we get in particular that $(s_1\alpha_1)^e$ must assume the same value for all $s_1\in S$ and are done, so assume that $\sigma$ is the transposition $(1,2)$. For $s_1,s_2\in S$, set $f(s_1,s_2):=s_1\alpha_1(s_2)\alpha_1\alpha_2\in\Aut(S)$ (which runs through the entire coset $S\alpha_1\alpha_2$ as $s_1,s_2$ run through $S$) and $g(s_1,s_2):=((s_1\alpha_1,s_2\alpha_2)(1,2))^2=(s_1\alpha_1s_2\alpha_2,s_2\alpha_2s_1\alpha_1)=(f(s_1,s_2),\tau_{s_1\alpha_1}^{-1}(f(s_1,s_2)))$. If $e$ is even, then we get by assumption that $g(s_1,s_2)^{e/2}$ is constant, whence the $(e/2)$-th and thus also the $e$-th power map is constant on the coset $S\alpha_1\alpha_2$, as required. And if $e$ is odd, we get from the assumed constancy of $((s_1\alpha_1,s_2\alpha_2)(1,2))^e=g(s_1,s_2)^{(e-1)/2}\cdot(s_1\alpha_1,s_2\alpha_2)(1,2)$ that $g(s_1,s_2)^{(e-1)/2}$ must be injective in $(s_1,s_2)$, in particular that $g(s_1,s_2)$ must be injective. However, the entries of $g(s_1,s_2)$ both are elements of the coset $S\alpha_1\alpha_2$ (since $\Out(S)$ is abelian), and they are conjugate under some element from the coset $S\alpha_1^{-1}$. Hence $g(s_1,s_2)$ can only assume $|S|^2$ pairwise distinct values if all elements of $S\alpha_1\alpha_2$ are conjugate to one another under elements of $S\alpha_1^{-1}$. In particular, no two distinct elements from $S\alpha_1^{-1}$ conjugate $\alpha_1\alpha_2$ to the same element, so that $\alpha_1\alpha_2$ centralizes no element from $S\setminus\{1\}$. However, this contradicts the fact that $\alpha_1\alpha_2$ has a nontrivial fixed point in $S$ (see \cite[Theorem]{Row95a}).
\end{proof}

In the rest of this section, we will prove Theorem \ref{mainTheo}(1), starting with the \enquote{infinite part}:

\begin{corollary}\label{oddCorollary}
Let $e\in\IZ$ be odd. Then the power word $x^e$ is multiplicity-bounding.
\end{corollary}

\begin{proof}
By Theorem \ref{powerWordTheo}, assume $e>0$ and note that it is sufficient to show that the $e$-th power map is non-constant on every coset of $S$ in $\Aut(S)$ for all nonabelian finite simple groups $S$ of one of the two forms $\PSL_2(p^L)$ or $\Suz(2^{2L-1})$. We proceed by contradiction.

If $S=\Suz(2^{2L-1})$, then the field automorphisms of $S$ form a set of representatives for the cosets of $S$ in $\Aut(S)$. Hence what we need to show boils down to proving that for no field automorphism $\sigma$ of $S$, it is the case that as $s$ runs through $S$, $(s\sigma)^e=s\sigma(s)\cdots\sigma^{e-1}(s)\sigma^e$ is always equal to $\sigma^e$. If that was the case, it would follow in particular that for all $s\in\Suz(2)\leq S$, $s^e=s\sigma(s)\cdots\sigma^{e-1}(s)=1$, whence $20=\exp(\Suz(2))\mid e$, a contradiction.

Now assume that $S=\PSL_2(p^L)$ for some prime $p$ and $L\in\IN^+$ with $p^L\geq 5$. For $p=2$, the field automorphisms of $S$ form a set of coset representatives for $S$ in $\Aut(S)$, and we can conclude as in the previous case, using that $\exp(\PSL_2(2))=\exp(\Sym_3)=6$, so assume that $p>2$. Fix a generator $\xi_{p^L}$ of $\IF_{p^L}^{\ast}$. Then, noting that $\Aut(S)=\Aut(\PSL_2(p^L))=\PGL_2(p^L)\rtimes\Gal(\IF_{p^L}/\IF_p)$, we find that the $2L$ elements of $\Aut(S)$ of the form

\[
r(\epsilon,\sigma):=\overline{\begin{pmatrix}\xi_{p^L}^{\epsilon} & 0 \\ 0 & 1\end{pmatrix}}\sigma,
\]

where $\epsilon\in\{0,1\}$ and $\sigma$ runs through the field automorphisms of $S$, form a set of representatives for the cosets of $S$ in $\Aut(S)$. That the $e$-th power map is non-constant on $Sr(0,\sigma)$ for all field automorphisms $\sigma$ follows as before, using that $\exp(\PSL_2(p))$ is even. It remains to show the assertion for $\epsilon=1$, for which we proceed by induction on $L$.

For $L=1$, the only field automorphism of $S$ is $\id_S$, and we have $\overline{\begin{pmatrix}\xi_p & 1 \\ 0 & 1\end{pmatrix}}\in Sr(1,\id_S)$, which has $e$-th power

\[
\overline{\begin{pmatrix}\xi_p^e & 1+\xi_p+\cdots+\xi_p^{e-1} \\ 0 & 1\end{pmatrix}}=\overline{\begin{pmatrix}\xi_p^e & \frac{\xi_p^e-1}{\xi_p-1} \\ 0 & 1\end{pmatrix}},
\]

which is distinct from $r(1,\id_S)^e=\begin{pmatrix}\xi_p^e & 0 \\ 0 & 1\end{pmatrix}$ because $\xi_p^e-1\not=0$.

Now assume that $L>1$. If $L$ is not a power of $2$, then we are done using the induction hypothesis and Lemma \ref{sufficientlyLem} with $H:=\PSL_2(p^{L/l})$, $l$ any odd prime divisor of $L$, which is conjugacy-uniquely maximal in $S$ by \cite[Theorem 2.1(o)]{Kin05a}. Hence assume that $L$ is a (nontrivial) power of $2$. Let $\sigma$ be any field automorphism of $S$, and let $K\in\{0,\ldots,L-1\}$ be such that $\sigma$ corresponds to the element $x\mapsto x^{p^K}$ of $\Gal(\IF_{p^L}/\IF_p)$. Then

\[
r(1,\sigma)^e=\overline{\begin{pmatrix}\xi_{p^L}^{1+p^K+\cdots+p^{(e-1)K}} & 0 \\ 0 & 1\end{pmatrix}}\sigma^e,
\]

and $Sr(1,\sigma)$ also contains the element $\overline{\begin{pmatrix}\xi_{p^L}^3 & 0 \\ 0 & 1\end{pmatrix}}$, which has $e$-th power

\[
\overline{\begin{pmatrix}\xi_{p^L}^{3(1+p^K+\cdots+p^{(e-1)K})} & 0 \\ 0 & 1\end{pmatrix}}\sigma^e,
\]

and this is distinct from $r(1,\sigma)^e$, since $1+p^K+\cdots+p^{(e-1)K}$ is odd and the multiplicative order of $\xi_{p^L}$, $p^L-1=(p^{L/2})^2-1$, is divisible by $8$.
\end{proof}

Now we turn to even exponents $e$, starting with those for which Theorem \ref{mainTheo}(1) asserts that their power word is \emph{not} multiplicity-bounding:

\begin{corollary}\label{evenCorollary1}
For $e\in\{8,12,16,18,24,30\}$, the power word $x^e$ is not multiplicity-bounding.
\end{corollary}

\begin{proof}
By Theorem \ref{powerWordTheo}, we need to find, for each of the six $e$, a coset $C$ of a nonabelian finite simple group $S$ in $\Aut(S)$ such that the $e$-th power map is constant on that coset. For $e=12,24$, take $S:=\PSL_2(5)=\Alt_5$ and $C:=\Sym_5\setminus\Alt_5$. For $e=8$, let $\sigma$ be the nontrivial field automorphism of $S:=\PSL_2(9)$, and take $C:=(\PGL_2(9)\setminus\PSL_2(9))\sigma$ (that all elements in $C$ have order a divisor of $8$ can be checked with GAP), and for $e=18$, let $\sigma$ be one of the two nontrivial field automorphisms of $S:=\PSL_2(8)$, and take $C:=S\sigma$ (and check that this does the job with GAP again). Finally, for $e=30$, just take $S:=C:=\Alt_5$.
\end{proof}

As mentioned after Theorem \ref{mainTheo}, the author used a GAP implementation of a certain algorithm to verify that for all even integers $e$ with $|e|\leq22$ not mentioned in Corollary \ref{evenCorollary1}, $x^e$ is multiplicity-bounding. We will now describe this algorithm; the basic idea is to use the equivalence with Theorem \ref{mainTheo}(3), of course, but it turns out that one can reduce the number of cases to be checked by the algorithm even further, which is vital to reduce the actual computation time to a reasonable amount. In what follows, let $e$ be any even and w.l.o.g.~positive integer.

\begin{itemize}
\item We first note that by the element structure of $\PGL_2(q)$ and $\Suz(2^{2L-1})$ (for the latter, see Suzuki's original paper \cite{Suz60a}), it is clear that

\[
\exp(\PSL_2(2^L))=\exp(\PGL_2(2^L))=2\cdot(2^{2L}-1),
\]

that $\exp(\PSL_2(p^L))=\frac{1}{4}p(p^{2L}-1)$ and

\[\lcm_{x\in\PGL_2(p^L)\setminus\PSL_2(p^L)}{\ord(x)}=\frac{1}{2}(p^{2L}-1)
\]

for odd primes $p$, and that

\[
\exp(\Suz(2^{2L-1}))=(2^{2L-1}+2^L+1)\cdot(2^{2L-1}-2^L+1)\cdot(2^{2L-1}-1)\cdot 4.
\]

Hence we may assume that $e$ is not divisible by any of these numbers for $L$ such that the corresponding group $\PSL_2(2^L)$, $\PSL_2(p^L)$ or $\Suz(2^{2L-1})$ is simple, since otherwise, $x^e$ is clearly not multiplicity-bounding.

\item Now let $L\in\{2,\ldots,e^2\}$ be a prime; in this bullet point, we consider constancy of the $e$-th power map on cosets of $S:=\PSL_2(2^L)$ in $\Aut(\PSL_2(2^L))=\PSL_2(2^L)\rtimes\Gal(\IF_{2^L}/\IF_2)$ (recall that it is sufficient to consider prime $L$ by Proposition \ref{mainReductionProp}(2)). Fix a field automorphism $\sigma$ of $S$, say $\sigma$ corresponds to the automorphism $x\mapsto x^{2^K}$, $K\in\{0,\ldots,L-1\}$, of $\IF_{2^L}$. By our assumption from the previous bullet point, the $e$-th power map is certainly non-constant on $S\sigma$ for $K=0$, so assume $K\geq 1$. The question is whether the expression $M\cdot\sigma(M)\cdots\sigma^{e-1}(M)$ assumes the same value on all $M\in S$. Now it certainly assumes $1_S$ for $M:=1_S$, and for $\xi_{2^L}$ a generator of $\IF_{2^L}^{\ast}$ and $M:=\overline{\begin{pmatrix}\xi_{2^L} & 0 \\ 0 & 1\end{pmatrix}}$, the expression becomes

\[
\overline{\begin{pmatrix}\xi_{2^L}^{\frac{2^{eK}-1}{2^K-1}} & 0 \\ 0 & 1\end{pmatrix}},
\]

which is equal to the identity element of $S$ if and only if $(2^K-1)(2^L-1)\mid 2^{eK}-1$, which implies $2^L-1\mid 2^{eK}-1$ and thus $L\mid eK$. Since $L$ is a prime and $1\leq K<L$, we get from this that $L\mid e$. Summing up, we see that for the $\PSL_2(2^L)$, \emph{it suffices to consider only those $L\in\{2,\ldots,e^2\}$ that are prime divisors of $e$}.

\item Next, let $L\in\{2,\ldots,e^2\}$ be either a prime or a power of $2$; we now consider constancy of the $e$-th power map on cosets of $S:=\PSL_2(3^L)$ in $\Aut(\PSL_2(3^L))=\PGL_2(3^L)\rtimes\Gal(\IF_{3^L}/\IF_3)$. Fix a field automorphism $\sigma$, say $\sigma$ corresponds to $x\mapsto x^{3^K}$, $K\in\{1,\ldots,L-1\}$, let $\xi_{3^L}$ be a generator of $\IF_{3^L}^{\ast}$, and set $\theta_{3^L}:=\xi_{3^L}^2$, a generator of the group of squares in $\IF_{3^L}^{\ast}$. With $\sigma$ fixed, we investigate whether the $e$-th power map is constant on one of the two cosets $S\overline{\begin{pmatrix}\theta_{3^L} & 0 \\ 0 & 1\end{pmatrix}}\sigma$ or $S\overline{\begin{pmatrix}\xi_{3^L} & 0 \\ 0 & 1\end{pmatrix}}\sigma$ of $S$ in $\Aut(S)$, which is equivalent to the constancy of $M\cdot\sigma(M)\cdots\sigma^{e-1}(M)$ for $M\in S\overline{\begin{pmatrix}\zeta_{3^L} & 0 \\ 0 & 1\end{pmatrix}}$, $\zeta\in\{\theta,\xi\}$ fixed. This coset of $S$ contains the two elements $\overline{\begin{pmatrix}\zeta_{3^L}^{\pm1} & 0 \\ 0 & 1\end{pmatrix}}$, and plugging them into the above expression yields

\[
\overline{\begin{pmatrix}\zeta^{\pm\frac{3^{eK}-1}{3^K-1}} & 0 \\ 0 & 1\end{pmatrix}},
\]

which are distinct elements of $\PGL_2(3^L)$ for both choices of $\zeta$ unless $(3^K-1)(3^L-1)\mid 4\cdot(3^{eK}-1)$, hence unless $3^L-1\mid 4\cdot(3^{eK}-1)$. We claim that (as in the previous case), this last divisibility relation implies $L\mid eK$. Indeed, this is clear for $L=2$ since $e$ is even, and for $L\geq 3$, we note that by Zsigmondy's theorem, $3^{\gcd(L,eK)}-1=\gcd(3^L-1,3^{eK}-1)=3^L-1$, since by assumption, $\gcd(3^L-1,3^{eK}-1)$ is divisible by all prime divisors of $3^L-1$. For prime $L$, this again implies $L\mid e$, so \emph{among the prime $L\in\{2,\ldots,e^2\}$, we only need to consider those that are prime divisors of $e$}. Moreover, for $L$ that are powers of $2$, we do not get a further restriction on $L$ from this, but a restriction on $K$: \emph{For each fixed $L=2^f\leq e^2$, we only need to consider those $K$ which are divisible by $2^{\max(0,f-\nu_2(e))}$.} In particular, for given $L=2^f$, the number of $K$ to consider is bounded from above by $2^{\nu_2(e)}$.

\item Now let $L\in\{2,\ldots,e^2\}$ be a power of $2$. With an argument analogous to the one for the prime $3$, we see that \emph{for each fixed $L=2^f\leq e^2$, we only need to consider those $K$ which are divisible by $2^{\max(0,f-\nu_2(e))}$} when checking the constancy of the $e$-th power map on either of the two cosets  $S\overline{\begin{pmatrix}\xi_{p^L} & 0 \\ 0 & 1\end{pmatrix}}\sigma$ or $S\overline{\begin{pmatrix}\xi_{p^L}^2 & 0 \\ 0 & 1\end{pmatrix}}\sigma$, $\xi_{p^L}$ a generator of $\IF_{p^L}^{\ast}$ and $\sigma$ the field automorphism corresponding to $x\mapsto x^{p^K}$.

\item Finally, let $L\in\{2,\ldots,4e^2\}$ be a prime. We study constancy of the $e$-th power map on cosets of $S:=\Suz(2^{2L-1})$ in $\Aut(S)=S\rtimes\Gal(\IF_{2^{2L-1}}/\IF_2)$. Fix a field automorphism $\sigma$ of $S$, say corresponding to $x\mapsto x^{2^K}$ for $K\in\{1,\ldots,2L-2\}$, and study the constancy of the expression $M\cdot\sigma(M)\cdots\sigma^{e-1}(M)$ on $S$. Consider Suzuki's original $4$-dimensional representation of $S$ over $\IF_{2^{2L-1}}$ from \cite{Suz60a} to view $S$ as a subgroup of $\GL_4(2^{2L-1})$. Then for any fixed generator $\xi_{2^{2L-1}}$ of $\IF_{2^{2L-1}}^{\ast}$, $S$ contains the matrix

\[
M:=\begin{pmatrix}\xi_{2^{2L-1}} & 0 & 0 & 0 \\ 0 & \xi^{2^L-1} & 0 & 0 \\ 0 & 0 & \xi^{1-2^L} & 0 \\ 0 & 0 & 0 & \xi^{-1}\end{pmatrix},
\]

which when plugged into the expression yields a matrix equal to the identity matrix if and only if $(2^K-1)(2^{2L-1}-1)\mid 2^{eK}-1$. Hence like before, we conlude that \emph{we only need to consider those $L\in\{2,\ldots,4e^2\}$ such that $2L-1$ is a prime divisor of $e$}.
\end{itemize}

The algorithm implemented by the author only checks the cases not eliminated by the above reduction arguments. In each of the cases, the algorithm chooses elements $M$ from a fixed coset of $\PSL_2(p^L)$ in $\PGL_2(p^L)$ resp.~from $\Suz(2^{2L-1})$ at random and checks whether $M\sigma(M)\cdots\sigma^{e-1}(M)$ is distinct from the value of that expression in some \enquote{standard} element. Note that while our reduction arguments heavily reduce the list of primes $L$ to consider, the $L$ that are powers of $2$ still need to be checked up to $e^2$. So, for example, for $e=22$, we need to do matrix computations \emph{inter alia} over $\IF_{19^{256}}$. To deal with this, the author first stored various primitive polynomials over prime fields in a list; as far as available, the polynomials from \cite[Supplement, S47--S50]{HM92a} were used, and the remaining needed primitive polynomials (such as one of degree $256$ over $\IF_{19}$) were found using a simple random search algorithm also implemented by the author in GAP.

With these preparations, the author's algorithm can be used to quickly verify the following result case by case, concluding the proof of Theorem \ref{mainTheo}(1):

\begin{corollary}\label{evenCorollary2}
For each even integer $e$ with $|e|\leq 22$ and $|e|\notin\{8,12,16,18\}$, the power word $x^e$ is multiplicity-bounding.\qed
\end{corollary}

\section{Systematic study of short words}\label{sec6}

In this section, we discuss how to prove with computer aid that all nonempty reduced words of length at most $8$ except for the power words of length $8$ are multiplicity-bounding. Actually, we will show the following stronger statement:

\begin{proposition}\label{lengthEightProp}
All nonempty reduced words of length at most $8$ which are not power words of length $8$ are very strongly multiplicity-bounding.
\end{proposition}

We will work with the following weakening of \enquote{weakly multiplicity-bounding} (compare with Proposition \ref{nikolovProp}(2)):

\begin{definition}\label{veryWeaklyDef}
A reduced word $w$ of length $l$ and maximum variable multiplicity $m$ is \emph{very weakly multiplicity-bounding} if and only if none of the coset word maps of $w$ on any of the finite simple groups of one of the two forms $\PSL_2(p^L)$, $2\leq p\leq m$, $L\leq ml$, or $\Suz(2^{2L-1})$, $2\leq L\leq 4ml$, is constant.
\end{definition}

We note the following simple result, which is useful for reduction arguments cutting down the number of words of a given length to consider:

\begin{lemma}\label{permSubLem}
Let $w=w(x_1,\ldots,x_d)$ be a very weakly multiplicity-bounding reduced word. Then for all $\epsilon_1,\ldots,\epsilon_d\in\{\pm1\}$ and all $\sigma\in\Sym_d$, the reduced word $w(x_{\sigma(1)}^{\epsilon_1},\ldots,x_{\sigma(d)}^{\epsilon_d})$ is also very weakly multiplicity-bounding. In particular, the inverse of $w$ and the \emph{mirror form of $w$} (obtained by reading $w$ in reverse order) are very weakly multiplicity-bounding.\qed
\end{lemma}

Of course, analogous statements with \enquote{weakly multiplicity-bounding} etc.~instead of \enquote{very weakly multiplicity-bounding} also hold true.

The subsequent lemma collects some further simple, but important facts used in the proof of Proposition \ref{lengthEightProp}:

\begin{lemma}\label{lengthEightLem}
The following statements hold:

\begin{enumerate}
\item Let $l\in\IN^+$, and assume that all reduced words of length $l$ (resp.~all such words except for the power words) are very weakly multiplicity-bounding. Then actually all reduced words of length $l$ (resp.~all such words except for the power words) are very strongly multiplicity-bounding.
\item Let $w$ be a reduced word of the form $w_1(y_1,\ldots,y_t)w_2(x_1,\ldots,x_d)w_3(y_1,\ldots,y_t)$, where the variable sets $\{y_1,\ldots,y_t\}$ and $\{x_1,\ldots,x_d\}$ are disjoint (in which case we also say that \emph{$w_2$ can be isolated in $w$}) and the word $w_2(x_1,\ldots,x_d)$ is very strongly multiplicity-bounding. Then $w$ is very strongly multiplicity-bounding.
\item Let $l\in\IN^+$, and assume that for all $n\in\IN^+$, $n<l$, all reduced words of length $n$ are very strongly multiplicity-bounding. Then:

\begin{enumerate}
\item Every reduced word of length $l$ in which some variable occurs with multiplicity at most $2$ is very strongly multiplicity-bounding.
\item Every reduced word of length $l$ in at least $\lfloor l/3\rfloor+1$ distinct variables is very strongly multiplicity-bounding.
\end{enumerate}
\end{enumerate}
\end{lemma}

\begin{proof}
For (1): By definition, a reduced word $w$ of length $l$ is very strongly multiplicity-bounding if and only if all its variations, which are also of length $l$ and are not power words if $w$ is not a power word, are weakly multiplicity-bounding. But for this, it is actually sufficient that all variations of $w$ are very weakly multiplicity-bounding, in view of Proposition \ref{nikolovProp}(2), Lemma \ref{equationSystemLem} and the fact that variations of variations of $w$ are equivalent (in the sense of Definition \ref{variationDef}(4)) to variations of $w$.

For (2): This is a generalization of Proposition \ref{examplesProp}, and the proof is also a generalization of its proof: A variation $v$ of $w$ is of the form $v_1(\vec{y})v_2(\vec{x})v_3(\vec{y})$ with $v_2(\vec{x})$ weakly multiplicity-bounding. Say $\vec{x}=(x^{(1)},\ldots,x^{(d')})$ and $\vec{y}=(y^{(1)},\ldots,y^{(t')})$. Now in an equation of the form

\[
v(x^{(1)}\alpha_1,\ldots,x^{(d')}\alpha_{d'},y^{(1)}\beta_1,\ldots,y^{(t')}\beta_{t'})=\gamma,
\]

$\alpha_1,\ldots,\alpha_{d'},\beta_1,\ldots,\beta_{t'},\gamma\in\Aut(S)$, $S$ a nonabelian finite simple group, fixed, with all variables ranging over $S$, the middle segment $v_2(x^{(1)}\alpha_1,\ldots,x^{(d')}\alpha_{d'})$ of the left-hand side can be isolated. Hence assuming that the equation is universally solvable over $S^{d'+t'}$ and fixing the values of $y^{(1)},\ldots,y^{(t')}$, it follows that some coset word map of $v_2$ on $S$ is constant, contradicting that $v_2$ is weakly multiplicity-bounding.

For (3,a): Let $w$ be a reduced word of length $l$ in which some variable occurs with multiplicity at most $2$. If there even is a variable occurring with multiplicity $1$ in $w$, apply Proposition \ref{examplesProp}(1). Otherwise, $w$ is of the form

\[
w_1(x_1,\ldots,x_{d-1})x_d^{\epsilon_1}w_2(x_1,\ldots,x_{d-1})x_d^{\epsilon_2}w_3(x_1,\ldots,x_{d-1})
\]

with $\epsilon_1,\epsilon_2\in\{\pm1\}$. If $w_2$ is empty,then necessarily $\epsilon_1=\epsilon_2$ and $x_d^{2\epsilon_1}$ can be isolated in $w$. But $x_d^{2\epsilon_1}$ is very strongly multiplicity-bounding by Corollary \ref{evenCorollary2} and Proposition \ref{examplesProp}(1), so we can apply statement (2) and are done. Finally, if $w_2$ is non-empty, it is very strongly multiplicity-bounding by assumption, so apply Proposition \ref{examplesProp}(2,3) to conclude the proof.

For (3,b): This follows from (3,a), since a reduced word of length $l$ in at least $\lfloor l/3\rfloor+1 > l/3$ distinct variables must contain some variable with multiplicity at most $2$.
\end{proof}

From Lemma \ref{lengthEightLem}, we immediately get the following approximation to Proposition \ref{lengthEightProp}:

\begin{lemma}\label{lengthFiveLem}
All nonempty reduced words of length at most $5$ are very strongly multiplicity-bounding.
\end{lemma}

\begin{proof}
We proceed by induction on $l=1,\ldots,5$ to show that all reduced words of length $l$ are very strongly multiplicity-bounding. For $l=1$, this follows from Proposition \ref{examplesProp}(1). Assuming that $l>1$ now, we deduce from Lemma \ref{lengthEightLem}(3,b) that all reduced words of length $l$ in which the number of distinct variables is larger than $\lfloor5/3\rfloor=1$ are very strongly multiplicity-bounding, in particular very weakly multiplicity-bounding. Hence by Corollaries \ref{oddCorollary} and \ref{evenCorollary2}, all reduced words of length $l$ are very weakly multiplicity-bounding, and we conclude by an application of Lemma \ref{lengthEightLem}(1).
\end{proof}

In the subsequent discussion concerning words of length $6$, $7$ or $8$, we will use the following notation: For positive integers $n$ and $k$, denote by $\Part(n,k)$ the number of $k$-tuples $(a_1,\ldots,a_k)$ of positive integers such that $a_1+\cdots+a_k=n$. Moreover, fix a set $\overline{\Part}(n,k)$ of representatives for the classes of the equivalence relation $\sim$ on $\Part(n,k)$, defined by $(a_1,\ldots,a_k)\sim(b_1,\ldots,b_k)$ if and only if either $a_i=b_i$ for all $i=1,\ldots,k$ or $a_i=b_{k-i+1}$ for all $i=1,\ldots,k$. Furthermore, set $\p(n,k):=|\Part(n,k)|$ and $\overline{\p}(n,k):=|\overline{\Part}(n,k)|$.

We now describe our computational approach to reduced words $w$ of length $6$, $7$ or $8$ to conclude the proof of Proposition \ref{lengthEightProp}. 

\begin{itemize}
\item In view of Lemma \ref{lengthEightLem}(1), our goal is to show successively for $l=6,7,8$ that all reduced words $w$ of length $l$ except for the power words $x^{\pm8}$ are very weakly multiplicity-bounding.
\item By Lemma \ref{lengthEightLem}(3,b), this is clear if the number of distinct variables occurring in $w$ is at least $\lfloor 8/3\rfloor+1=3$, and the power word case (up to length $7$) is also clear by Corollaries \ref{oddCorollary} and \ref{evenCorollary2}. Hence we may restrict ourselves to the $2$-variable case, where $w$ is an alternating product of powers of variables $x$ and $y$, and by Lemma \ref{permSubLem}, we may also assume w.l.o.g.~that the first two variable powers in $w$ have positive exponent.
\item Denote by $b_w$ the number of variable power factors of $w$ (for example, $b_w=4$ for $w=x^2y^4x^{-1}y$). By Lemma \ref{lengthEightLem}(2), Corollaries \ref{oddCorollary} and \ref{evenCorollary2} and the induction hypothesis, we may assume $b_w\geq 4$. Fix $b_w\in\{4,\ldots,l\}$.
\item If $b_w$ is even, say $b_w=2k$, then $w$ is of the form $x^{e_1}y^{e_2}\cdots x^{e_{2k-1}}y^{e_{2k}}$. By the induction hypothesis and Lemma \ref{lengthEightLem}(3,a), we may assume w.l.o.g.~that $\mu_w(x),\mu_w(y)\geq\max(3,k)$, so that $\max(3,k)\leq\mu_w(x)\leq l-\max(3,k)$. Moreover, note that if we replace $w$ by its inverse and, if necessary, make substitutions of the form $x\mapsto x^{-1}$, $y\mapsto y^{-1}$, we can arrange for $y$ to be the first variable appearing in $w$ and the first two powers in $w$ still having positive exponent. Therefore, we may assume w.l.o.g.~that $\mu_w(x)\in\{\max(3,k),\ldots,\lfloor l/2\rfloor\}$. Fix $\mu_w(x)$ in this range, so that $\mu_w(y)=l-\mu_w(x)$ also gets fixed. Then we only need to consider words of the above form in which $(e_1,|e_3|,\ldots,|e_{2k-1}|)\in\Part(\mu_w(x),k)$ and $(e_2,|e_4|,\ldots,|e_{2k}|)\in\Part(l-\mu_w(x),k)$; these are $\p(\mu_w(x),k)\cdot\p(l-\mu_w(x),k)\cdot 2^{b_w-2}$ many words altogether.
\item If $b_w$ is odd, say $b_w=2k+1$, then $w$ is of the form $x^{e_1}y^{e_2}\cdots x^{e_{2k-1}}y^{e_{2k}}x^{e_{2k+1}}$. As before, we may assume w.l.o.g.~that $\mu_w(x)\geq\max(3,k+1)=k+1$, $\mu_w(y)\in\max(3,k)$. Fix $\mu_w(x)\in\{k+1,\ldots,l-\max(3,k)\}$. For the further reductions, we make a case distinction:

\begin{itemize}
\item If $\mu_w(x)=k+1$, we know that $(|e_1|,|e_3|,\ldots,|e_{2k+1}|)=(1,1,\ldots,1)$, and this property is not destroyed when replacing $w$ by its mirror form. Hence we may assume w.l.o.g.~that $(e_2,|e_4|,\ldots,|e_{2k}|)\in\overline{\Part}(l-\mu_w(x),k)$. Altogether, there are $\overline{p}(l-\mu_w(x),k)\cdot 2^{b_w-2}$ many words to consider in this subcase.
\item If $\mu_w(x)>k+1$, we may assume w.l.o.g.~that $(e_1,|e_3|,\ldots,|e_{2k+1}|)\in\overline{\Part}(\mu_w(x),k+1)$ and $(e_2,|e_4|,\ldots,|e_{2k}|)\in\Part(l-\mu_w(x),k)$, so there are $\overline{p}(\mu_w(x),k+1)\cdot\p(l-\mu_w(x),k)\cdot 2^{b_w-2}$ many words to consider in this subcase.
\end{itemize}
\end{itemize}

Table \ref{table1} at the end of this section gives an overview of the number of words we must consider in total; altogether, there are $40$ words of length $6$, $168$ words of length $7$ and $628$ words of length $8$ for which we need to check that they are very weakly multiplicity-bounding to conclude the proof of Proposition \ref{lengthEightProp}. We did this with the aid of GAP \cite{GAP4}.

\begin{table}[h]\caption{Number of words to consider for fixed $(l,b_w,\mu_w(x))$}\label{table1}
\begin{center}
\begin{tabular}{|c|c|c|c|}\hline
$l$ & $b_w$ & $\mu_w(x)$ & number of words \\ \hline
6 & 4 & 3 & 16 \\ \hline
6 & 5 & 3 & 8 \\ \hline
6 & 6 & 3 & 16 \\ \hline
7 & 4 & 3 & 24 \\ \hline
7 & 5 & 3 & 16 \\ \hline
7 & 5 & 4 & 48 \\ \hline
7 & 6 & 3 & 48 \\ \hline
7 & 7 & 4 & 32 \\ \hline
8 & 4 & 3 & 32 \\ \hline
8 & 4 & 4 & 36 \\ \hline
8 & 5 & 3 & 16 \\ \hline
8 & 5 & 4 & 48 \\ \hline
8 & 5 & 5 & 64 \\ \hline
8 & 6 & 3 & 96 \\ \hline
8 & 6 & 4 & 144 \\ \hline
8 & 7 & 4 & 64 \\ \hline
8 & 7 & 5 & 64 \\ \hline
8 & 8 & 4 & 64 \\ \hline
\end{tabular}
\end{center}
\end{table}

\section{Concluding remarks}\label{sec7}

We conclude this paper with two open problems and questions for further research on multiplicity-bounding words.

The first question concerns our results on power words. By Corollary \ref{oddCorollary}, we know that all power words of odd length are multiplicity-bounding, but the following interesting question is open:

\begin{question}\label{evenQues}
Do there exist infinitely many (positive) \emph{even} integers $e$ such that the power word $x^e$ is multiplicity-bounding?
\end{question}

Generally speaking, it would be desirable to have a better understanding of the possible element orders in cosets of nonabelian finite simple groups in their automorphism groups and in particular formulas for $\lcm_{s\in S}{\ord(s\alpha)}$, where $S$ is a nonabelian finite simple group and $\alpha$ an automorphism of $S$. With regard to Question \ref{evenQues}, by Theorem \ref{powerWordTheo}, it would suffice to have such formulas for $S$ one of $\PSL_2(q)$ or $\Suz(2^{2L-1})$; note that for $\alpha\in S$, the above least common multiple is just $\exp(S)$, for which there are known formulas at least in those two cases.

The second open question which we would like to raise is the following:

\begin{question}\label{algorithmQues}
Does there exist an algorithm which on input a reduced word $w$ decides whether $w$ is multiplicity-bounding?
\end{question}

We remark that it is not difficult to see by Lemma \ref{nonUnivLem} that a reduced word $w$ is \emph{not} multiplicity-bounding if and only if the following holds: There exists a nonabelian finite simple group $S$, a strictly monotonically increasing sequence $(n_i)_{i\in\IN}$ of positive integers and a constant $C\geq0$ such that for all $i\in\IN$, there exist $\vec{\alpha_1}^{(i)},\ldots,\vec{\alpha_d}^{(i)},\vec{\beta}^{(i)}\in\Aut(S^{n_i})$ such that in the system of $n_i$ equations as described in Lemma \ref{equationSystemLem}, at most $C$ of the equations are not universally solvable. Hence a possible approach to Question \ref{algorithmQues} consists of a more careful study of the possible combinations of equations in such systems (studying particularly the $\chi_j$, which are word map evaluations in the permutations $\sigma_k\in\Sym_n$ associated with the $\vec{\alpha_k}$).

\end{document}